\documentclass[10pt,reqno]{amsart}
\usepackage{graphicx}
\usepackage{amsfonts}
\usepackage{amssymb}
\usepackage{amsmath}
\usepackage{amsxtra}
\usepackage{latexsym}
\usepackage{epstopdf}
\usepackage{mathrsfs}
\usepackage{esint}
\usepackage{setspace}
\usepackage{mathrsfs,galois,booktabs,ctable,threeparttable,tabularx,algorithm,algorithmic}
\usepackage{enumitem}
\usepackage{listings}

\newtheorem{theorem}{Theorem}[section]
\newtheorem{lemma}[theorem]{Lemma}

\newtheorem{proposition}[theorem]{Proposition}
\newtheorem{corollary}[theorem]{Corollary}

\theoremstyle{definition}

\newtheorem{example}{Example}[section]

\newtheorem{assumption}{Assumption}[section]

\theoremstyle{remark}
\newtheorem{remark}{Remark}[section]

\numberwithin{equation}{section}

\begin{document}

\title[Proximal Alternating direction method of multipliers]
{On convergence rates of proximal alternating direction method of multipliers}

\author{Qinian Jin}

\address{Mathematical Sciences Institute, Australian National
University, Canberra, ACT 2601, Australia}
\email{qinian.jin@anu.edu.au} \curraddr{}


\def\ep{\varepsilon}
\date{March 3, 2023}

\begin{abstract}
In this paper we consider from two different aspects the proximal alternating direction method 
of multipliers (ADMM) in Hilbert spaces. We first consider the application of the proximal ADMM 
to solve well-posed linearly constrained two-block separable convex minimization problems in 
Hilbert spaces and obtain new and improved non-ergodic convergence rate results, including linear 
and sublinear rates under certain regularity conditions. We next consider proximal ADMM as a regularization 
method for solving linear ill-posed inverse problems in Hilbert spaces. When the data is 
corrupted by additive noise, we establish, under a benchmark source condition, a convergence 
rate result in terms of the noise level when the number of iteration is properly chosen.
\end{abstract}


\def\p{\partial}
\def\d{\delta}
\def\l{\langle}
\def\r{\rangle}
\def\a{\alpha}
\def\la{\lambda}
\def\ep{\varepsilon}
\def\V{\mathcal V}
\def\X{\mathcal X}
\def\Y{\mathcal Y}
\def\Z{\mathcal Z}
\def\H{\mathcal H}
\def\N{\mathcal N}
\def\C{\mathcal C}

\maketitle

\section{\bf Introduction}
\setcounter{equation}{0}

The alternating direction method of multipliers (ADMM) was introduced and developed in 
the 1970s by Glowinski, Marrocco \cite{GM1975} and Gabay, Mercier \cite{GM1976} for the 
numerical solutions of partial differential equations. Due to its decomposability and 
superior flexibility, ADMM and its variants have gained renewed interest in recent 
years and have been widely used for solving large-scale optimization problems that 
arise in signal/image processing, statistics, machine learning, inverse problems and 
other fields, see \cite{BPCPE2011,GBO2010,JJLW2016}. Because of their popularity, 
many works have been devoted to the analysis of ADMM and its variants, see 
\cite{BPCPE2011,DY2016,EP1992,G1983,HY2012,LM1979,ZBO2011} for instance.
In this paper we will devote to deriving convergence rates of ADMM in two aspects: 
its applications to solve well-posed convex optimization problems and its use to 
solve linear ill-posed inverse problems as a regularization method. 

In the first part of this paper we consider ADMM for solving linearly constrained two-block 
separable convex minimization problems. Let $\X$, $\Y$ and $\Z$ be real Hilbert spaces with 
possibly infinite dimensions. We consider the convex minimization problem of the form
\begin{align}\label{prob}
\begin{split}
\mbox{minimize } & \quad H(x, y) := f(x) + g(y) \\
\mbox{subject to } & \quad A x + By =c,
\end{split}
\end{align}
where $c \in \Z$, $A: \X \to \Z$ and $B: \Y \to \Z$ are bounded linear operators, 
and $f: \X\to (-\infty, \infty]$ and $g: \Y \to (-\infty, \infty]$ are proper, 
lower semi-continuous, convex functions. The classical ADMM solves (\ref{prob}) 
approximately by constructing an iterative sequence via alternatively 
minimizing the augmented Lagrangian function 
$$
{\mathscr L}_\rho (x, y, z) := f(x) + g(y) + \l \la, A x + B y -c \r + \frac{\rho}{2} \|A x + B y - c\|^2
$$
with respect to the primal variables $x$ and $y$ and then updating the dual 
variable $\la$; more precisely, starting from an initial guess $y^0\in \Y$ 
and $\la^0\in \Z$, an iterative sequence $\{(x^k, y^k, \la^k)\}$ is defined by 
\begin{align}\label{admm}
\begin{split}
x^{k+1} &= \arg\min_{x\in \X} \left\{f(x) + \l \la^k, A x\r + \frac{\rho}{2} \|A x+B y^k -c\|^2 \right\},\\
y^{k+1} &= \arg\min_{y\in \Y} \left\{g(y) + \l \la^k, B y\r + \frac{\rho}{2} \|A x^{k+1}+By -c\|^2 \right\},\\
\la^{k+1} &= \la^k + \rho (A x^{k+1}+B y^{k+1}-c),
\end{split}
\end{align}
where $\rho>0$ is a given penalty parameter. The implementation of (\ref{admm}) requires 
to determine $x^{k+1}$ and $y^{k+1}$ by solving two convex minimization problems during 
each iteration. Although $f$ and $g$ may have special structures so that their proximal 
mappings are easy to be determined, solving the minimization problems in (\ref{admm}) 
in general is highly nontrivial due to the appearance of the terms $\|A x\|^2 $ and 
$\|B y\|^2$. In order to avoid this implementation issue, one may consider to add 
certain proximal terms to the $x$-subproblem and $y$-subproblem in (\ref{admm}) to remove 
the terms $\|A x\|^2$ and $\|B y\|^2$. For any bounded linear positive semi-definite 
self-adjoint operator $D$ on a real Hilbert space $\H$, we will use the notation
$$
\|u\|_D^2 : = \l z, D u\r, \quad \forall u \in \H.
$$
By taking two bounded linear positive semi-definite self-adjoint operators $P: \X \to \X$ 
and $Q: \Y \to \Y$, we may add the terms $\frac{1}{2} \|x-x^k\|_P^2$ and 
$\frac{1}{2} \|y-y^k\|_Q^2$ to the $x$- and $y$-subproblems in (\ref{admm}) respectively 
to obtain the following proximal alternating direction method of multipliers 
(\cite{B2017,DY2016b,HY2012,HY2015,JJLW2017,ZBO2011})
\begin{align}\label{alg1}
\begin{split}
x^{k+1} &= \arg\min_{x\in \X} \left\{ f(x) + \l \la^k, A x\r + \frac{\rho}{2} \left\|A x + B y^k -c \right\|^2
+ \frac{1}{2} \|x-x^k\|_P^2\right\},\\
y^{k+1} &= \arg\min_{y\in \Y} \left\{g(y) + \l \la^k, B y\r + \frac{\rho}{2} \left\|A x^{k+1}+By-c \right\|^2
+ \frac{1}{2} \|y-y^k\|_Q^2\right\},\\[1ex]
\lambda^{k+1} &= \la^k + \rho (A x^{k+1} + B y^{k+1} -c).
\end{split}
\end{align}
\noindent 
The advantage of (\ref{alg1}) over (\ref{admm}) is that, with wise choices of $P$ and $Q$, 
it is possible to remove the terms $\|A x\|^2$ and $\|B y\|^2$ and thus make the determination 
of $x^{k+1}$ and $y^{k+1}$ much easier. 

In recent years, various convergence rate results have been established for ADMM and its 
variants in either ergodic or non-ergodic sense. In \cite{HY2012,LMZ2015}
the ergodic convergence rate 
\begin{align}\label{ergodic}
|H(\bar x^k, \bar y^k)-H_*| = O\left(\frac{1}{k}\right)  \quad \mbox{and} \quad 
\|A \bar x^k+B \bar y^k -c\|= O \left(\frac{1}{k}\right)
\end{align}
has been derived in terms of the objective error and the constraint error, where $H_*$ 
denotes the minimum value of (\ref{prob}), $k$ denotes the number of iteration, and 
$$
\bar x^k := \frac{1}{k} \sum_{j=1}^k x^j \qquad \mbox{and} \qquad 
\bar y^k := \frac{1}{k} \sum_{j=1}^k y^j
$$
denote the ergodic iterates of $\{x^k\}$ and $\{y^k\}$ respectively; see also 
\cite[Theorem 15.4]{B2017}. A criticism on ergodic result is that it may fail to 
capture the feature of the sought solution of the underlying problem because 
ergodic iterate has the tendency to average out the expected property and thus 
destroy the feature of the solution. This is in particular undesired in sparsity
optimization and low-rank learning. In contrast, the non-ergodic iterate tends 
to share structural properties with the solution of the underlying problem. 
Therefore, the use of non-ergodic iterates becomes more favorable in practice.
In \cite{HY2015} a non-ergodic convergence rate has been derived for the 
proximal ADMM (\ref{alg1}) with $Q=0$ and the result reads as  
\begin{align}\label{ner0}
\|x^{k+1} - x^k\|_P^2 + \|B(y^{k+1}-y^k)\|^2 + \|\la^{k+1}-\la^k\|^2 
= o \left(\frac{1}{k}\right).
\end{align}
By exploiting the connection with the Douglas-Rachford splitting algorithm, 
the non-ergodic convergence rate
\begin{align}\label{ner1}
|H(x^k, y^k)-H_*| = o\left(\frac{1}{\sqrt{k}}\right)  \quad \mbox{and} \quad 
\|A x^k+B y^k -c\|= o \left(\frac{1}{\sqrt{k}}\right)
\end{align}
in terms of the objective error and the constraint error has been established 
in \cite{DY2016} for the ADMM (\ref{admm}) and an example has been provided to 
demonstrate that the estimates in (\ref{ner1}) are sharp. However, the derivation 
of (\ref{ner1}) in \cite{DY2016} relies on some unnatural technical conditions
involving the convex conjugate of $f$ and $g$, see Remark \ref{Rk2.1}. 
Note that the estimate (\ref{ner0}) implies the second estimate in (\ref{ner1}), 
however it does not imply directly the first estimate in (\ref{ner1}).  In Section 
\ref{sect2} we will show, by a simpler argument, that similar estimate as in 
(\ref{ner0}) can be established for the proximal ADMM (\ref{alg1}) with arbitrary 
positive semi-definite $Q$. Based on this result and some additional properties of 
the method, we will further show that the non-ergodic rate (\ref{ner1}) holds for the 
proximal ADMM (\ref{alg1}) with arbitrary positive semi-definite $P$ and $Q$. 
Our result does not require any technical conditions as assumed in \cite{DY2016}. 

In order to obtain faster convergence rates for the proximal ADMM (\ref{alg1}), 
certain regularity conditions should be imposed. In finite dimensional situation, 
a number of linear convergence results have been established. In \cite{DY2016b} some 
linear convergence results of the proximal ADMM have been provided under a number 
of scenarios involving the strong convexity of $f$ and/or $g$, the Lipschitz 
continuity of $\nabla f$ and/or $\nabla g$, together with further full row/column 
rank assumptions on $A$ and/or $B$. Under a bounded metric subregularity condition, 
in particular under the assumption that both $f$ and $g$ are convex piecewise 
linear-quadratic functions, a global linear convergence rate has been established 
in \cite{YH2016} for the proximal ADMM (\ref{alg1}) with 
\begin{align}\label{yhc}
P:=\tau_1 I - \rho A^*A \succ 0 \quad \mbox{and} \quad Q:= \tau_2 I - \rho B^*B\succ 0,
\end{align}
where $A^*$ and $B^*$ denotes the adjoints of $A$ and $B$ respectively; 
the condition (\ref{yhc}) plays an essential role in the convergence analysis in 
\cite{YH2016}. We will derive faster convergence rates for the proximal ADMM (\ref{alg1}) 
in the general Hilbert space setting. To this end, we need first to consider 
the weak convergence of $\{(x^k, y^k, \la^k)\}$ and demonstrate that every 
weak cluster point of this sequence is a KKT point of (\ref{prob}). 
This may not be an issue in finite dimensions. However, this is nontrivial 
in infinite dimensional spaces because extra care is required to dealing with 
weak convergence. In \cite{BS2017} the weak convergence of the proximal ADMM 
(\ref{alg1}) has been considered by transforming the method into a proximal point 
method and the result there requires restrictive conditions, see 
\cite[Lemma 3.4 and Theorem 3.1]{BS2017}. These restrictive 
conditions have been weakened later in \cite{Sun2019} by using machinery from 
the maximal monotone operator theory. We will explore the structure of the 
proximal ADMM and show by an elementary argument that every weak cluster point 
of $\{(x^k, y^k, \la^k)\}$ is indeed a KKT point of (\ref{prob}) without any 
additional conditions. We will then consider the linear convergence of the 
proximal ADMM under a bounded metric subregularity condition and obtain the 
linear convergence for any positive semi-definite $P$ and $Q$; in particular, we 
obtain the linear convergence of $|H(x^k, y^k) - H_*|$ and $\|A x^k + B y^k - c\|$.
We also consider deriving convergence rates under a bounded H\"{o}lder metric 
subregularity condition which is weaker than the bounded metric subregularity. 
This weaker condition holds if both $f$ and $g$ are semi-algebraic functions and 
thus a wider range of applications can be covered. We show that, under a bounded 
H\"{o}lder metric subrigularity condition, among other things the convergence 
rates in (\ref{ner1}) can be improved to 
$$
\|A x^k + B y^k - c\| = O(k^{-\beta}) \quad \mbox{ and } \quad 
|H(x^k, y^k) - H_*| = O(k^{-\beta})
$$
for some number $\beta >1/2$; the value of $\beta$ depends on the properties 
of $f$ and $g$. To further weaken the bounded (H\"{o}lder) metric subregularity 
assumption, we introduce an iteration based error bound condition which is an 
extension of the one in \cite{LYZZ2018} to the general proximal ADMM (\ref{alg1}). 
It is interesting to observe that this error bound condition holds under any 
one of the scenarios proposed in \cite{DY2016b}. Hence, we provide a unified analysis 
for deriving convergence rates under the bounded (H\"{o}lder) metric subregularity 
or the scenarios in \cite{DY2016b}. Furthermore, we extend the scenarios in 
\cite{DY2016b} to the general Hilbert space setting and demonstrate that some 
conditions can be weakened and the convergence result can be strengthened; 
see Theorem \ref{thm2.11}. 

In the second part of this paper, we consider using ADMM as a regularization method to 
solve linear ill-posed inverse problems in Hilbert spaces. Linear 
inverse problems have a wide range of applications, including medical imaging, 
geophysics, astronomy, signal processing, and more (\cite{EHN1996,G1984,N2001}).
We consider linear inverse problems of the form
\begin{align}\label{ip1.1}
Ax = b, \quad x\in \C, 
\end{align}
where $A : \X \to \H$ is a compact linear operator between two Hilbert spaces $\X$ 
and $\H$, $\C$ is a closed convex set in $\X$, and $b \in \mbox{Ran}(A)$, the range 
of $A$. In order to find a solution of (\ref{ip1.1}) with desired properties, a priori 
available information on the sought solution should be incorporated into the problem. 
Assume that, under a suitable linear transform $L$ from $\X$ to another Hilbert 
spaces $\Y$ with domain $\mbox{dom}(L)$, the feature of the sought solution can be 
captured by a proper convex penalty function $f : \Y\to (-\infty, \infty]$. One may 
consider instead of (\ref{ip1.1}) the constrained optimization problem
\begin{align}\label{ip1.2}
\min\{f (Lx): Ax = b, \ x\in \C, \ x\in \mbox{dom}(L)\}. 
\end{align}
A challenging issue related to the numerical resolution of (\ref{ip1.2}) is 
its ill-posedness in the sense that the solution of (\ref{ip1.2}) does not depend 
continuously on the data and thus a small perturbation on data can lead to a large 
deviation on solutions. In practical applications, the exact data $b$ is usually 
unavailable, instead only a noisy data $b^\d$ is at hand with 
$$
\|b^\d - b\| \le \d
$$
for some small noise level $\d>0$. To overcome ill-posedness, regularization methods 
should be introduced to produce reasonable approximate solutions; one may refer to 
\cite{BO2004,EHN1996,Jin2022,RS2006} for various regularization methods. 

The common use of ADMM to solve (\ref{ip1.2}) with noisy data $b^\d$ first 
considers the variational regularization 
\begin{align}\label{vr.admm}
\min_{x\in \C} \left\{\frac{1}{2} \|A x - b^\d\|^2 + \a f(Lx)\right\},
\end{align}
then uses the splitting technique to rewrite (\ref{vr.admm}) into the form (\ref{prob}), 
and finally applies the ADMM procedure to produce approximate solutions. The parameter 
$\a>0$ is the so-called regularization parameter which should be adjusted carefully to 
achieve reasonable good performance; consequently one has to run ADMM to solve (\ref{vr.admm}) 
for many different values of $\a$, which can be time consuming.

In \cite{JJLW2016,JJLW2017} the ADMM has been considered to solve (\ref{ip1.2}) 
directly to reduce the computational load. Note that (\ref{ip1.2}) can be written as  
\begin{align*}
\left\{\begin{array}{lll}
\min f(y) + \iota_{\C}(x)\\
\mbox{subject to } A z = b, \ L z - y = 0, \ z - x = 0, \ z \in \mbox{dom}(L),
\end{array}\right.
\end{align*}
where $\iota_{\C}$ denotes the indicator function of $\C$. With the noisy data $b^\d$ 
we introduce the augmented Lagrangian function 
\begin{align*}
{\mathscr L}_{\rho_1, \rho_2, \rho_3}(z, y, x, \la, \mu, \nu) 
&:= f(y) + \iota_{\C}(x) + \l \la, A z - b^\d\r + \l \mu, L z - y\r + \l \nu, z - x \r \\
& \quad \, + \frac{\rho_1}{2} \|A z - b^\d\|^2 + \frac{\rho_2}{2} \|L z - y\|^2 
+ \frac{\rho_3}{2} \|z -x\|^2,
\end{align*}
where $\rho_1$, $\rho_2$ and $\rho_3$ are preassigned positive numbers. The proximal 
ADMM proposed in \cite{JJLW2017} for solving (\ref{ip1.2}) then takes the form 
\begin{align}\label{PADMM2}
\begin{split}
& z^{k+1} = \arg\min_{z\in \mbox{dom}(L)} 
\left\{{\mathscr L}_{\rho_1, \rho_2, \rho_3}(z, y^k, x^k, \la^k, \mu^k, \nu^k) + \frac{1}{2} \|z-z^k\|_Q^2\right\},\\
& y^{k+1} = \arg\min_{y\in \Y} 
\left\{{\mathscr L}_{\rho_1, \rho_2, \rho_3}(z^{k+1}, y, x^k, \la^k, \mu^k, \nu^k) \right\},\\
& x^{k+1} = \arg\min_{x\in \X} 
\left\{{\mathscr L}_{\rho_1, \rho_2, \rho_3}(z^{k+1}, y^{k+1}, x, \la^k, \mu^k, \nu^k)\right\}, \\
& \la^{k+1} = \la^k + \rho_1 (A z^{k+1} - b^\d), \\
& \mu^{k+1} = \mu^k + \rho_2 (L z^{k+1} - y^{k+1}), \\
& \nu^{k+1} = \nu^k + \rho_3 (z^{k+1} - x^{k+1}),
\end{split}
\end{align}
where $Q$ is a bounded linear positive semi-definite self-adjoint operator. 
The method (\ref{PADMM2}) is not a 3-block ADMM. Note that the 
variables $y$ and $x$ are not coupled in 
${\mathscr L}_{\rho_1, \rho_2, \rho_3}(z, y, x, \la, \mu, \nu)$. Thus, $y^{k+1}$ 
and $x^{k+1}$ can be updated simultaneously, i.e. 
\begin{align*}
(y^{k+1}, x^{k+1}) = \arg\min_{y\in \Y, x \in \X} 
\left\{{\mathscr L}_{\rho_1, \rho_2, \rho_3}(z^{k+1}, y, x, \la^k, \mu^k, \nu^k) \right\}. 
\end{align*}
This demonstrates that (\ref{PADMM2}) is a 2-block proximal ADMM. 

It should be highlighted that all well-established convergence results on proximal ADMM 
for well-posed optimization problems are not applicable to (\ref{PADMM2}) directly. Note 
that (\ref{PADMM2}) uses the noisy data $b^\d$. If the convergence 
theory for well-posed optimization problems could be applicable, one would obtain 
a solution of the perturbed problem
\begin{align}\label{perturb}
\min\left\{f(Lx): A x = b^\d, \ x \in \C, \ x \in \mbox{dom}(L)\right\}
\end{align}
of (\ref{ip1.2}). Because $A$ is compact, it is very likely that $b^\d\not\in \mbox{Ran}(A^*)$
and thus (\ref{perturb}) makes no sense as the feasible set is empty. Even if 
$b^\d \in \mbox{Ran}(A^*)$ and (\ref{perturb}) has a solution, this solution 
could be far away from the solution of (\ref{ip1.2}) because of the ill-posedness. 

Therefore, if (\ref{PADMM2}) is used to solve (\ref{ip1.2}), better result can not be expected even if 
larger number of iterations are performed. In contrast, like all other iterative 
regularization methods, when (\ref{PADMM2}) is used to solve (\ref{ip1.2}), it shows the semi-convergence 
property, i.e., the iterate becomes close to the sought solution at the beginning; 
however, after a critical number of iterations, the iterate leaves the sought solution 
far away as the iteration proceeds. Thus, properly terminating the iteration is 
important to produce acceptable approximate solutions. One may hope 
to determine a stopping index $k_\d$, depending on $\d$ and/or $b^\d$, such that 
$\|x^{k_\d}-x^\dag\|$ is as small as possible and $\|x^{k_\d}-x^\dag\|\to 0$ 
as $\d\to 0$, where $x^\dag$ denotes the solutio of (\ref{ip1.2}).
This has been done in our previous work \cite{JJLW2016,JJLW2017} in which 
early stopping rules have been proposed for the method (\ref{PADMM2}) to render it 
into a regularization method and numerical results have been reported to demonstrate 
the nice performance. However, the work in \cite{JJLW2016,JJLW2017} does not provide 
convergence rates, i.e. the estimate on $\|x^{k_\d} - x^\dag\|$ in terms of $\d$.
Deriving convergence rates for iterative regularization methods involving general convex 
regularization terms is a challenging question and only a limited number of results are 
available. In order to derive a convergence rate of a regularization method for 
ill-posed problems, certain source condition should be imposed on the sought solution. 
In Section \ref{sect3}, under a benchmark source condition on the sought solution, 
we will provide a partial answer to this question by establishing a convergence rate 
result for (\ref{PADMM2}) if the iteration is terminated by an {\it a priori} 
stopping rule. 

We conclude this section with some notation and terminology. Let $\V$ be a real 
Hilbert spaces. We use $\l\cdot, \cdot\r$ and $\|\cdot\|$ to denote its inner 
product and the induced norm. We also use ``$\to$" and ``$\rightharpoonup$" to 
denote the strong convergence and weak convergence respectively. For a function 
$\varphi: \V \to (-\infty, \infty]$ its domain is defined as 
$\mbox{dom}(\varphi) := \{x\in \V: \varphi(x) <\infty\}$. If 
$\mbox{dom}(\varphi)\ne \emptyset$, $\varphi$ is called proper. For a proper 
convex function $\varphi: \V \to (-\infty, \infty]$, its modulus of convexity, 
denoted by $\sigma_\varphi$, is defined to be the largest number $c$ such that 
$$
\varphi(t x + (1-t) y) + c t(1-t) \|x-y\|^2 \le t \varphi(x) + (1-t) \varphi(y)
$$
for all $x, y \in \mbox{dom}(\varphi)$ and $0\le t\le 1$. We always have 
$\sigma_\varphi\ge 0$. If $\sigma_\varphi>0$, $\varphi$ is called strongly 
convex. For a proper convex function $\varphi: \V \to (-\infty, \infty]$, we use 
$\p \varphi$ to denote its subdifferential, i.e. 
$$
\p \varphi(x) :=\{\xi \in \V: \varphi(y) \ge \varphi(x) + \l \xi, y-x\r 
\mbox{ for all } y \in \V\}, \quad x \in \V. 
$$
Let $\mbox{dom}(\p \varphi) :=\{x\in \V: \p \varphi(x) \ne \emptyset\}$. It is 
easy to see that 
$$
\varphi(y) - \varphi(x) - \l \xi, y-x\r \ge \sigma_\varphi \|y-x\|^2 
$$
for all $y\in \V$, $x\in \mbox{dom}(\p \varphi)$ and $\xi \in \p \varphi(x)$ which 
in particular implies the monotonicity of $\p \varphi$, i.e. 
$$
\l \xi - \eta, x - y\r \ge 2 \sigma_\varphi \|x-y\|^2 
$$
for all $x, y \in \mbox{dom}(\p \varphi)$, $\xi \in \p \varphi(x)$ and $\eta \in \p \varphi(y)$.

\section{\bf Proximal ADMM for convex optimization problems}\label{sect2}
\setcounter{equation}{0}

In this section we will consider the proximal ADMM (\ref{alg1}) for solving the 
linearly constrained convex minimization problem (\ref{prob}). For the convergence 
analysis, we will make the following standard assumptions.

\begin{assumption}\label{Ass1}
$\X$, $\Y$ and $\Z$ are real Hilbert spaces, $A: \X \to \Z$ and $B: \Y \to \Z$ are 
bounded linear operators, $P: \X \to \X$ and $Q: \Y \to \Y$ are bounded linear 
positive semi-definite self-adjoint operators, and $f: \X\to (-\infty, \infty]$ 
and $g: \Y \to (-\infty, \infty]$ are proper, lower semi-continuous, convex functions.
\end{assumption}

\begin{assumption}\label{Ass2}
The problem (\ref{prob}) has a Karush-Kuhn-Tucker (KKT) point, i.e. there exists 
$(\bar x, \bar y, \bar \la) \in \X\times \Y \times \Z$ such that
$$
-A^* \bar \la \in \p f(\bar x), \quad -B^* \bar \la \in \p g (\bar y), \quad 
A \bar x + B \bar y =c.
$$
\end{assumption}

It should be mentioned that, to guarantee the proximal ADMM (\ref{alg1}) 
to be well-defined, certain additional conditions need to be imposed to ensure that 
the $x$- and $y$-subproblems do have minimizers. Since the well-definedness can 
be easily seen in concrete applications, to make the presentation more succinct we will 
not state these conditions explicitly. 

By the convexity of $f$ and $g$, it is easy to see that, for any KKT point 
$(\bar x, \bar y, \bar \la)$ of (\ref{prob}), there hold
\begin{align*}
0 &\le f(x) - f(\bar x) + \l\bar \la, A(x-\bar x)\r, \quad \forall x \in \X,\\
0 &\le g(y) - g(\bar y) + \l\bar \la, B(y-\bar y)\r, \quad \forall y \in \Y.
\end{align*}
Adding these two equations and using $A \bar x+ B\bar y-c =0$, it follows that
\begin{align}\label{9.18.11}
0 \le H(x, y) - H(\bar x, \bar y) + \l \bar \la, A x + B y-c \r, \quad 
\forall (x, y) \in \X \times \Y. 
\end{align}
This in particular implies that $(\bar x, \bar y)$ is a solution of (\ref{prob}) 
and thus $H_*:= H(\bar x, \bar y)$ is the minimum value of (\ref{prob}). 

Based on Assumptions \ref{Ass1} and \ref{Ass2} we will analyze the proximal ADMM 
(\ref{alg1}). For ease of exposition, we set $\widehat Q := \rho B^* B + Q$ and define 
\begin{align*}
G u:= (P x, \widehat Q y, \la/\rho), \quad \forall 
u:= (x, y, \la) \in \X \times \Y \times \Z
\end{align*}
which is a bounded linear positive semi-definite self-adjoint operator on
$\X\times \Y \times \Z$. Then, for any $u:=(x, y,\la) \in \X\times \Y \times \Z$ we have
$$
\|u\|_G^2:=\l u, G u\r = \|x\|_P^2 + \|y\|_{\widehat Q}^2 + \frac{1}{\rho} \|\la\|^2.
$$
For the sequence $\{u^k:=(x^k, y^k, \la^k)\}$ defined by the proximal ADMM (\ref{alg1}),
we use the notation
\begin{align*} 
\Delta x^k := x^k-x^{k-1}, \ \ \Delta y^k := y^k-y^{k-1}, \ \ \Delta\la^k := \la^k-\la^{k-1}, \ \  
\Delta u^k := u^k - u^{k-1}.
\end{align*}
We start from the first order optimality conditions on $x^{k+1}$ and $y^{k+1}$ which 
by definition can be stated as 
\begin{align}\label{fg}
\begin{split}
-A^* \la^k - \rho A^* (A x^{k+1} + B y^k -c) - P (x^{k+1}-x^k) & \in \p f(x^{k+1}), \\
-B^* \la^k - \rho B^* (A x^{k+1} + B y^{k+1} -c) - Q(y^{k+1}-y^k) & \in \p g(y^{k+1}). 
\end{split}
\end{align}
By using $\la^{k+1}=\la^k + \rho (A x^{k+1} + B y^{k+1} -c)$ we may rewrite (\ref{fg}) as 
\begin{align}\label{11.11.1}
\begin{split}
-A^* (\la^{k+1} - \rho B \Delta y^{k+1}) -P \Delta x^{k+1} &\in \p f(x^{k+1}), \\
-B^* \la^{k+1} - Q \Delta y^{k+1} &\in \p g(y^{k+1})
\end{split}
\end{align}
which will be frequently used in the following analysis. We first prove the following important 
result which is inspired by \cite[Lemma 3.1]{HY2012} and \cite[Theorem 15.4]{B2017}. 

\begin{proposition}\label{prop9.18}
Let Assumption \ref{Ass1} hold. Then for the proximal ADMM (\ref{alg1}) there holds
\begin{align*}
& \sigma_f \|x^{k+1} - x\|^2 + \sigma_g \|y^{k+1} - y\|^2 \\
& \le  H(x, y) - H(x^{k+1}, y^{k+1}) + \l \la^{k+1} - \rho B \Delta y^{k+1}, A x + B y -c \r \\
& \quad \, - \l \la, A x^{k+1} + B y^{k+1}-c\r + \frac{1}{2} \left(\|u^k-u\|_G^2  - \|u^{k+1}-u\|_G^2\right) \\
& \quad \, -\frac{1}{2\rho} \|\Delta \la^{k+1} - \rho B \Delta y^{k+1}\|^2 
-\frac{1}{2} \|\Delta x^{k+1}\|_P^2- \frac{1}{2} \|\Delta y^{k+1}\|_Q^2
\end{align*}
for all $u := (x, y, \la) \in \X \times \Y \times \Z$, where $\sigma_f$ and $\sigma_g$ denote the modulus of 
convexity of $f$ and $g$ respectively. 
\end{proposition}

\begin{proof}
Let $\tilde \la^{k+1} := \la^{k+1} - \rho B \Delta y^{k+1}$. By using (\ref{11.11.1}) and
the convexity of $f$ and $g$ we have for any $(x, y, \la) \in \X \times \Y \times \Z$ that 
\begin{align*}
& \sigma_f \|x^{k+1} - x\|^2 + \sigma_g \|y^{k+1} - y\|^2 \\
& \le  f(x) - f(x^{k+1}) + \l \la^{k+1} -\rho B \Delta y^{k+1}, A (x-x^{k+1})\r 
+ \l P\Delta x^{k+1}, x -x^{k+1}\r \\
& \quad \, + g(y) - g(y^{k+1}) + \l \la^{k+1}, B(y-y^{k+1})\r 
+ \l Q \Delta y^{k+1}, y-y^{k+1}\r \displaybreak[0]\\
& =  H(x, y) - H(x^{k+1}, y^{k+1}) + \l \tilde \la^{k+1}, A(x-x^{k+1}) + B(y - y^{k+1})\r \\
& \quad \, + \l P\Delta x^{k+1}, x -x^{k+1}\r + \l \widehat Q \Delta y^{k+1}, y-y^{k+1}\r \displaybreak[0]\\
& = H(x, y) - H(x^{k+1}, y^{k+1}) + \l \tilde \la^{k+1}, A x + B y -c\r \\
& \quad \, - \l \la, A x^{k+1} + B y^{k+1} -c \r 
+ \l \la - \tilde \la^{k+1}, A x^{k+1} + B y^{k+1} - c\r \\
& \quad \, + \l P\Delta x^{k+1}, x -x^{k+1}\r + \l \widehat Q \Delta y^{k+1}, y-y^{k+1}\r.
\end{align*}
Since $\rho(A x^{k+1} + B y^{k+1} - c) = \Delta \la^{k+1}$ we then obtain 
\begin{align*}
& \sigma_f \|x^{k+1} - x\|^2 + \sigma_g \|y^{k+1} - y\|^2 \\
& \le  H(x, y) - H(x^{k+1}, y^{k+1}) + \l \tilde \la^{k+1}, A x + By -c\r 
- \l \la, A x^{k+1} + B y^{k+1} -c \r \\
& \quad \, + \frac{1}{\rho} \l \la - \la^{k+1}, \Delta \la^{k+1}\r 
+ \frac{1}{\rho}\l \la^{k+1} - \tilde \la^{k+1}, \Delta \la^{k+1} \r \\
& \quad \, + \l P\Delta x^{k+1}, x -x^{k+1}\r 
+ \l \widehat Q \Delta y^{k+1}, y-y^{k+1}\r.
\end{align*}
By using the polarization identity and the definition of $G$, it follows that
\begin{align*}
& \sigma_f \|x^{k+1} - x\|^2 + \sigma_g \|y^{k+1} - y\|^2 \\
& \le  H(x, y) - H(x^{k+1}, y^{k+1}) + \l \tilde \la^{k+1}, A x + B y -c\r 
- \l \la, A x^{k+1} + B y^{k+1} -c \r \displaybreak[0]\\
& \quad \, + \frac{1}{2\rho} \left(\|\la^k-\la\|^2 - \|\la^{k+1} -\la\|^2 
- \|\Delta \la^{k+1}\|^2 \right) \displaybreak[0]\\
& \quad \, - \frac{1}{2\rho} \left(\|\la^k - \tilde \la^{k+1}\|^2 
- \|\la^{k+1} - \tilde \la^{k+1}\|^2 - \|\Delta \la^{k+1}\|^2\right) \displaybreak[0]\\
& \quad \, + \frac{1}{2} \left(\|x^k-x\|_P^2 -\|x^{k+1}-x\|_P^2 
- \|\Delta x^{k+1}\|_P^2 \right) \displaybreak[0]\\
& \quad \, + \frac{1}{2} \left(\|y^k-y\|_{\widehat Q}^2 - \|y^{k+1}-y\|_{\widehat Q}^2 
-\|\Delta y^{k+1}\|_{\widehat Q}^2 \right) \displaybreak[0]\\
& = H(x, y) - H(x^{k+1}, y^{k+1}) + \l \tilde \la^{k+1}, A x + B y -c\r 
- \l \la, A x^{k+1} + B y^{k+1} -c \r \\
& \quad \, + \frac{1}{2} \left(\|u^k-u\|_G^2 - \|u^{k+1} - u\|_G^2 \right)
-\frac{1}{2\rho} \left(\|\la^k - \tilde \la^{k+1}\|^2 - \|\la^{k+1} - \tilde \la^{k+1}\|^2\right) \\
& \quad \, - \frac{1}{2} \|\Delta x^{k+1}\|_P^2 - \frac{1}{2} \|\Delta y^{k+1}\|_{\widehat Q}^2. 
\end{align*}
Using the definition of $\tilde  \la^{k+1}$ gives
\begin{align*}
\la^k - \tilde \la^{k+1}  = - \Delta \la^{k+1} + \rho B \Delta y^{k+1}, \quad 
\la^{k+1} -\tilde  \la^{k+1} = \rho B \Delta y^{k+1}.
\end{align*}
Therefore
\begin{align*}
& \sigma_f \|x^{k+1} - x\|^2 + \sigma_g \|y^{k+1} - y\|^2 \\
& \le  H(x, y) - H(x^{k+1}, y^{k+1}) + \l \tilde \la^{k+1}, A x + B y -c\r 
- \l \la, A x^{k+1} + B y^{k+1} -c \r \\
& \quad \, + \frac{1}{2} \left(\|u^k-u\|_G^2 - \|u^{k+1} - u\|_G^2 \right)
-\frac{1}{2\rho} \|\Delta \la^{k+1} - \rho B \Delta y^{k+1}\|^2 \\
& \quad \, - \frac{1}{2} \|\Delta x^{k+1}\|_P^2 - \frac{1}{2} \|\Delta y^{k+1}\|_{\widehat Q}^2
+ \frac{\rho}{2} \|B \Delta y^{k+1}\|^2. 
\end{align*}
Since $\rho \|B \Delta y^{k+1}\|^2 - \|\Delta y^{k+1}\|_{\widehat Q}^2 = - \|\Delta y^{k+1}\|_Q^2$, 
we thus complete the proof. 
\end{proof}

\begin{corollary}\label{cor2}
Let Assumption \ref{Ass1} and Assumption \ref{Ass2} hold and let $\bar u :=(\bar x, \bar y, \bar \la)$ 
be any KKT point of (\ref{prob}). Then for the proximal ADMM (\ref{alg1}) there holds
\begin{align}\label{admm.111}
 \sigma_f \|x^{k+1} - \bar x\|^2 + \sigma_g \|y^{k+1} - \bar y\|^2 
 & \le H_* - H(x^{k+1}, y^{k+1}) - \l \bar \la, A x^{k+1} + B y^{k+1} -c \r  \nonumber \\
& \quad \, + \frac{1}{2} \left(\|u^k-\bar u\|_G^2 - \|u^{k+1} - \bar u\|_G^2 \right)
\end{align}
for all $k\ge 0$. Moreover, the sequence $\{\|u^k-\bar u\|_G^2\}$ is monotonically 
decreasing.
\end{corollary}

\begin{proof}
By taking $u = \bar u$ in Proposition \ref{prop9.18} and using 
$A \bar x + B \bar y - c =0$ we immediately obtain (\ref{admm.111}). According to 
(\ref{9.18.11}) we have 
$$
H(x^{k+1}, y^{k+1}) - H_* + \l \bar \la, A x^{k+1} + B y^{k+1} -c \r \ge 0.   
$$
Thus, from (\ref{admm.111}) we can obtain 
\begin{align}\label{eq:cor.1}
\sigma_f \|x^{k+1} - \bar x\|^2 + \sigma_g \|y^{k+1} - \bar y\|^2 
\le \frac{1}{2} \left(\|u^k-\bar u\|_G^2 - \|u^{k+1} - \bar u\|_G^2 \right)
\end{align}
which implies the monotonicity of the sequence $\{\|u^k - \bar u\|_G^2\}$. 
\end{proof}

We next show that $\|\Delta u^k\|_G^2 = o(1/k)$ as $k\to \infty$. This result for
the proximal ADMM (\ref{alg1}) with $Q =0$ has been established in \cite{HY2015}
based on a variational inequality approach. We will establish this result for the 
proximal ADMM (\ref{alg1}) with general bounded linear positive semi-definite 
self-adjoint operators $P$ and $Q$ by a simpler argument. 

\begin{lemma}\label{lem:mono}
Let Assumption \ref{Ass1} hold. For the proximal ADMM (\ref{alg1}), the sequence $\{\|\Delta u^k\|_G^2\}$ 
is monotonically decreasing.
\end{lemma}

\begin{proof}
By using (\ref{11.11.1}) and the monotonicity of $\p f$ and $\p g$, we can obtain
\begin{align*}
0 \le & \left\l -A^*(\Delta \la^{k+1} - \rho B \Delta y^{k+1} + \rho B \Delta y^k) - P \Delta x^{k+1} + P \Delta x^k, \Delta x^{k+1}\right\r \\
& + \left\l - B^* \Delta \la^{k+1}-Q \Delta y^{k+1} 
+ Q\Delta y^k, \Delta y^{k+1} \right\r \\
= & - \l \Delta \la^{k+1}, A \Delta x^{k+1} + B \Delta y^{k+1} \r + \rho \l B (\Delta y^{k+1} -\Delta y^k), A \Delta x^{k+1}\r \\
& - \l P (\Delta x^{k+1} - \Delta x^k), \Delta x^{k+1}\r 
- \l Q(\Delta y^{k+1} - \Delta y^k), \Delta y^{k+1}\r.
\end{align*}
Note that
\begin{align*}
A \Delta x^{k+1} + B \Delta y^{k+1} = \frac{1}{\rho} (\Delta \la^{k+1} -\Delta \la^k).
\end{align*}
We therefore have
\begin{align*}
0 & \le - \frac{1}{\rho} \l \Delta \la^{k+1}, \Delta \la^{k+1} -\Delta \la^k\r 
- \rho \l B (\Delta y^{k+1} -\Delta y^k), B \Delta y^{k+1}\r \\
& \quad \, + \l B(\Delta y^{k+1} - \Delta y^k), \Delta \la^{k+1} - \Delta \la^k\r \\
& \quad \, - \l P (\Delta x^{k+1} - \Delta x^k), \Delta x^{k+1}\r - \l Q(\Delta y^{k+1} - \Delta y^k), \Delta y^{k+1}\r.
\end{align*}
By the polarization identity we then have
\begin{align*}
0 & \le \frac{1}{2\rho} \left(\|\Delta \la^k\|^2 - \|\Delta \la^{k+1}\|^2 - \|\Delta\la^k -\Delta \la^{k+1}\|^2 \right) \\
& \quad \, + \frac{\rho}{2} \left(\|B \Delta y^k\|^2 - \|B \Delta y^{k+1}\|^2 - \|B(\Delta y^k-\Delta y^{k+1})\|^2 \right)\\
& \quad \, + \frac{1}{2} \left(\|\Delta x^k \|_P^2 - \|\Delta x^{k+1}\|_P^2 - \|\Delta x^k-\Delta x^{k+1}\|_P^2\right)\\
& \quad \, + \frac{1}{2} \left(\|\Delta y^k\|_Q^2 - \|\Delta y^{k+1}\|_Q^2 - \|\Delta y^k-\Delta y^{k+1}\|_Q^2 \right)\\
& \quad \, + \l B(\Delta y^{k+1} - \Delta y^k), \Delta \la^{k+1} - \Delta \la^k\r.
\end{align*}
With the help of the definition of $G$, we obtain
\begin{align*}
0 & \le \|\Delta u^k\|_G^2 - \|\Delta u^{k+1}\|_G^2 - \|\Delta x^k-\Delta x^{k+1}\|_P^2 - \|\Delta y^k-\Delta y^{k+1}\|_Q^2 \\
& \quad \, - \frac{\rho}{2} \left\|B(\Delta y^{k+1} -\Delta y^k) - \frac{1}{\rho} (\Delta \la^{k+1} -\Delta \la^k) \right\|^2
\end{align*}
which completes the proof. 
\end{proof}

\begin{lemma}\label{prop9.20}
Let Assumptions \ref{Ass1} and \ref{Ass2} hold and let $\bar u:=(\bar x, \bar y, \bar \la)$ 
be any KKT point of (\ref{prob}). For the proximal ADMM (\ref{alg1}) there holds 
\begin{align*} 
& \|\Delta u^{k+1}\|_G^2 \le \left(\|u^k-\bar u\|_G^2 + \|\Delta y^k\|_Q^2\right) 
-\left(\|u^{k+1}-\bar u\|_G^2 + \|\Delta y^{k+1}\|_Q^2\right)
\end{align*}
for all $k \ge 1$.
\end{lemma}

\begin{proof}
 We will use (\ref{11.11.1}) together with $-A^* \bar \la\in \p f(\bar x)$ 
and $-B^* \bar \la \in \p g(\bar y)$. By using the monotonicity of $\p f$ and $\p g$
we have
\begin{align*}
0 & \le \left\l -A^* (\la^{k+1} -\bar \la - \rho B \Delta y^{k+1}) - P \Delta x^{k+1}, 
x^{k+1}-\bar x\right\r \\
& \quad \, + \left\l - B^* (\la^{k+1} - \bar \la) -Q \Delta y^{k+1}, y^{k+1} - \bar y\right\r \\
& = \l \bar \la-\la^{k+1}, A x^{k+1} + B y^{k+1} -c \r 
+ \rho \l B \Delta y^{k+1}, A (x^{k+1}-\bar x)\r \\
& \quad \, - \l P \Delta x^{k+1}, x^{k+1}-\bar x\r - \l Q \Delta y^{k+1}, y^{k+1} - \bar y\r.
\end{align*}
By virtue of $\rho(A x^{k+1} + B y^{k+1} - c\r = \Delta \la^{k+1}$ we further have
\begin{align*}
0 & \le \frac{1}{\rho} \l \bar \la -\la^{k+1}, \Delta \la^{k+1} \r 
- \rho \l B \Delta y^{k+1}, B(y^{k+1} - \bar y) \r
+ \l B \Delta y^{k+1}, \Delta \la^{k+1}\r \\
& \quad \, - \l P \Delta x^{k+1}, x^{k+1}-\bar x\r - \l Q \Delta y^{k+1}, y^{k+1} - \bar y\r.
\end{align*}
By using the second equation in (\ref{11.11.1}) and the monotonicity of $\p g$ we have
\begin{align*}
0 & \le \left\l -B^* \Delta \la^{k+1} - Q \Delta y^{k+1} + Q \Delta y^k, \Delta y^{k+1}\right\r \\
& = - \l \Delta \la^{k+1}, B \Delta y^{k+1}\r - \l Q(\Delta y^{k+1} -\Delta y^k), \Delta y^{k+1}\r
\end{align*}
which shows that
$$
 \l \Delta \la^{k+1}, B \Delta y^{k+1}\r \le - \l Q(\Delta y^{k+1} -\Delta y^k), \Delta y^{k+1}\r.
$$
Therefore
\begin{align*}
0\le & \frac{1}{\rho} \l \bar \la-\la^{k+1}, \Delta \la^{k+1} \r 
- \l \widehat Q \Delta y^{k+1}, y^{k+1} - \bar y\r - \l P \Delta x^{k+1}, x^{k+1}-\bar x\r\\
& - \l Q(\Delta y^{k+1} -\Delta y^k), \Delta y^{k+1}\r.
\end{align*}
By using the polarization identity we then obtain
\begin{align*}
0\le & \frac{1}{2\rho} \left(\l \la^k-\bar \la\|^2 -\|\la^{k+1} -\bar \la\|^2 
-\|\Delta \la^{k+1}\|^2 \right) \\
& + \frac{1}{2} \left( \|y^k-\bar y\|_{\widehat Q}^2 - \|y^{k+1}-\bar y\|_{\widehat Q}^2 
-\|\Delta y^{k+1}\|_{\widehat Q}^2\right)\\
& + \frac{1}{2} \left(\|x^k-\bar x\|_P^2 - \|x^{k+1}-\bar x\|_P^2 
- \|\Delta x^{k+1}\|_P^2 \right)  \\
& + \frac{1}{2} \left(\|\Delta y^k\|_Q^2 - \|\Delta y^{k+1}\|_Q^2 
- \|\Delta y^{k+1} -\Delta y^k\|_Q^2 \right).
\end{align*}
Recalling the definition of $G$ we then complete the proof. 
\end{proof}

\begin{proposition}\label{lem3}
Let Assumption \ref{Ass1} and Assumption \ref{Ass2} hold. Then for the proximal 
ADMM (\ref{alg1}) there holds $\|\Delta u^k\|_G^2 = o(1/k)$ as $k\to \infty$.
\end{proposition}

\begin{proof}
Let $\bar u$ be a KKT point of (\ref{prob}). From Lemma \ref{prop9.20} 
it follows that
\begin{align}\label{9.23.1}
\sum_{j=1}^k \|\Delta u^{j+1}\|_G^2
& \le \sum_{j=1}^k \left(\left(\|u^j-\bar u\|_G^2 + \|\Delta y^j\|_Q^2\right)
 - \left(\|u^{j+1}-\bar u\|_G^2+\|\Delta y^{j+1}\|_Q^2\right)\right) \nonumber\\
& \le \|u^1-\bar u\|_G^2 +\|\Delta y^1 \|_Q^2
\end{align}
for all $k \ge 1$. By Lemma \ref{lem:mono}, $\{\|\Delta u^{j+1}\|_G^2\}$ is monotonically 
decreasing. Thus
\begin{align}\label{9.23.2}
\left(\frac{k}{2} +1\right)\|\Delta u^{k+1}\|_G^2 \le \sum_{j=[k/2]}^k \|\Delta u^{j+1}\|_G^2,
\end{align}
where $[k/2]$ denotes the largest integer $\le k/2$. Since (\ref{9.23.1}) shows that
$$
\sum_{j=1}^\infty \|\Delta u^{j+1}\|_G^2 <\infty,
$$
the right hand side of (\ref{9.23.2}) must converge to $0$ as $k\to \infty$. 
Thus $(k+1) \|\Delta u^{k+1}\|_G^2 =o(1)$ and hence $\|\Delta u^{k}\|_G^2 =o(1/k)$ 
as $k\to \infty$. 
\end{proof}

As a byproduct of Proposition \ref{lem3} and Corollary \ref{cor2}, we can prove the 
following non-ergodic convergence rate result for the proximal ADMM (\ref{alg1}) in terms 
of the objective error and the constraint error.

\begin{theorem}\label{thm1}
Let Assumption \ref{Ass1} and Assumption \ref{Ass2} hold. Consider the proximal ADMM 
(\ref{alg1}) for solving (\ref{prob}). Then
\begin{align}\label{ner}
|H(x^k, y^k)-H_*| = o\left(\frac{1}{\sqrt{k}}\right)  \quad \mbox{and} \quad 
\|A x^k+B y^k -c\|= o \left(\frac{1}{\sqrt{k}}\right)
\end{align}
as $k \to \infty$.
\end{theorem}

\begin{proof}
Since  
\begin{align}\label{eq:sbl.2}
\rho (A x^{k} + B y^{k}-c) = \Delta \la^{k} \quad \mbox{and} \quad 
\|\Delta \la^{k}\|^2 \le \rho \|\Delta u^{k}\|_G^2
\end{align}
we may use Proposition \ref{lem3} to obtain the estimate 
$\|A x^{k}+B y^{k} -c\| = o (1/\sqrt{k})$ as $k\rightarrow \infty$. 

In the following we will focus on deriving the estimate of $|H(x^k, y^k) - H_*|$.  
Let $\bar u := (\bar x, \bar y, \bar \la)$ be a KKT point of (\ref{prob}). By using 
(\ref{admm.111}) we have 
\begin{align}\label{10.30.1}
H(x^{k}, y^{k}) - H_* 
& \le - \l \bar \la, A x^{k}+By^{k}-c\r 
+ \frac{1}{2} \left(\|u^{k-1}-\bar u\|_G^2 - \|u^{k}-\bar u\|_G^2\right) \nonumber \\
& = -\frac{1}{\rho} \l \bar \la, \Delta \la^{k}\r -\l u^{k-1}-\bar u, G \Delta u^{k}\r 
- \frac{1}{2} \|\Delta u^{k}\|_G^2 \nonumber \\
& \le \frac{\|\bar \la\|}{\rho} \|\Delta \la^{k}\| + \|u^{k-1}-\bar u\|_G \|\Delta u^{k}\|_G.
\end{align}
By virtue of the monotonicity of $\{\|u^k-\bar u\|_G^2\}$ given in Corollary \ref{cor2} 
we then obtain 
\begin{align*}
H(x^{k}, y^{k}) - H_*
& \le \frac{\|\bar \la\|}{\rho} \|\Delta \la^{k}\| 
+ \|u^0-\bar u\|_G \|\Delta u^{k}\|_G \nonumber \\
& \le \left(\|u^0-\bar u\|_G + \frac{\|\bar \la\|}{\sqrt{\rho}}\right) \|\Delta u^{k}\|_G.
\end{align*}
On the other hand, by using (\ref{9.18.11}) we have
\begin{align*}
H(x^{k}, y^{k}) - H_*
& \ge -\l \bar \la, A x^{k} + B y^{k} -c\r = -\frac{1}{\rho} \l \bar \la, \Delta \la^{k}\r \\
& \ge -\frac{\|\bar \la\|}{\rho} \|\Delta \la^{k}\| 
\ge - \frac{\|\bar \la\|}{\sqrt{\rho}} \|\Delta u^{k}\|_G.
\end{align*}
Therefore
\begin{align}\label{eq:sbl}
\left|H(x^{k}, y^{k}) - H_*\right|
 \le \left(\|u^0-\bar u\|_G + \frac{\|\bar \la\|}{\sqrt{\rho}}\right) \|\Delta u^{k}\|_G.
\end{align}
Now we can use Proposition \ref{lem3} to conclude the proof. 
\end{proof}

\begin{remark}\label{Rk2.1}
By exploiting the connection between the Douglas-Rachford splitting algorithm and 
the classical ADMM (\ref{admm}), the non-ergodic convergence rate (\ref{ner}) has 
been established in \cite{DY2016} for the classical ADMM (\ref{admm}) under the 
conditions that 
\begin{align}\label{DR1}
\mbox{zero}(\p d_f + \p d_g) \ne \emptyset
\end{align}
and 
\begin{align}\label{DR2}
\p d_f = A^*\circ \p f^* \circ A, \qquad \p d_g = B^*\circ \p g^* \circ B - c,
\end{align}
where $d_f (\la) := f^*(A^* \la)$ and $d_g(\la) := g^*(B^*\la)-\l \la, c\r$ with $f^*$ 
and $g^*$ denoting the convex conjugates of $f$ and $g$ respectively. The conditions 
(\ref{DR1}) and (\ref{DR2}) seems strong and unnatural because they are posed on 
the convex conjugates $f^*$ and $g^*$ instead of $f$ and $g$ themselves. 
In Theorem \ref{thm1} we establish the non-ergodic convergence rate (\ref{ner}) for 
the proximal ADMM (\ref{alg1}) with any positive semi-definite $P$ and $Q$ without 
requiring the conditions (\ref{DR1}) and (\ref{DR2}) and therefore our result extends 
and improves the one in \cite{DY2016}.
\end{remark}

Next we will consider establishing faster convergence rates under suitable regularity 
conditions. As a basis, we first prove the following result
which tells that any weak cluster point of $\{u^k\}$ is a KKT point of (\ref{prob}). 
This result can be easily established for ADMM in finite-dimensional spaces, however 
it is nontrivial for the proximal ADMM (\ref{alg1}) in infinite-dimensional Hilbert 
spaces due to the required treatment of weak convergence; Proposition \ref{prop9.18} 
plays a crucial role in our proof. 

\begin{theorem}\label{thm2:ADMM}
Let Assumption \ref{Ass1} and Assumption \ref{Ass2} hold. Consider the sequence 
$\{u^k:=(x^k, y^k, \la^k)\}$ generated by the proximal ADMM (\ref{alg1}). Assume 
$\{u^k\}$ is bounded and let $u^\dag :=(x^\dag, y^\dag, \la^\dag)$ be a 
weak cluster point of $\{u^k\}$. Then $u^\dag$ is a KKT point of (\ref{prob}).
Moreover, for any weak cluster point $u^*$ of $\{u^k\}$ there holds 
$
\|u^*-u^\dag\|_G =0. 
$
\end{theorem}

\begin{proof}
We first show that $u^\dag$ is a KKT point of (\ref{prob}). 
According to Propositon \ref{lem3} we have $\|\Delta u^k\|_G^2 \to 0$ which means 
\begin{align}\label{9.18.12}
\Delta \la^k \to 0, \quad P \Delta x^k \to 0, \quad B \Delta y^k \to 0, \quad Q\Delta y^k\to 0
\end{align}
as $k\to \infty$. According to Theorem \ref{thm1} we also have 
\begin{align}\label{admm-lc.1}
A x^k + B y^k -c \to 0 \quad \mbox{and} \quad H(x^k, y^k) \to H_* \quad \mbox{as } k \to \infty.
\end{align}
Since $u^\dag$ is a weak cluster point of the sequence $\{u^k\}$, there exists a subsequence 
$\{u^{k_j}:= (x^{k_j}, y^{k_j},\la^{k_j})\}$ of $\{u^k\}$ such that $u^{k_j} \rightharpoonup u^\dag$ 
as $j \to \infty$. By using the first equation in (\ref{admm-lc.1}) we immediately obtain 
\begin{align}\label{admm-lc.43}
A x^\dag + B y^\dag -c =0.
\end{align}
By using Proposition \ref{prop9.18} with $k = k_j -1$ we have for any 
$u := (x, y, \la) \in \X \times \Y \times \Z$ that
\begin{align}\label{admm-lc.41}
0 &\le H(x, y) - H(x^{k_j}, y^{k_j}) + \l\la^{k_j} - \rho B \Delta y^{k_j}, A x + B y - c\r \nonumber \\
& \quad - \l \la, A x^{k_j} + B y^{k_j} -c\r 
+ \frac{1}{2} \left(\|u^{k_j-1} - u\|_G^2 - \|u^{k_j} - u\|_G^2\right). 
\end{align}
According to Corollary \ref{cor2}, $\{\|u^k\|_G\}$ is bounded. Thus we may use 
Proposition \ref{lem3} to conclude 
$$
\left|\|u^{k_j-1} - u\|_G^2 - \|u^{k_j} - u\|_G^2\right| 
\le \left(\|u^{k_j-1} - u\|_G + \|u^{k_j} - u\|_G\right) \|\Delta u^{k_j}\|_G \to 0
$$
as $j \to \infty$. Therefore, by taking $j \to \infty$ in (\ref{admm-lc.41}) and using 
(\ref{9.18.12}), (\ref{admm-lc.1}) and $\la^{k_j} \rightharpoonup \la^\dag$ we can obtain 
\begin{align}\label{admm-lc.44}
0 \le H(x, y) - H_* + \l \la^\dag, A x + B y -c\r
\end{align}
for all $(x, y) \in \X \times \Y$. Since $f$ and $g$ are convex and lower semi-continuous, 
they are also weakly lower semi-continuous (see \cite[Chapter 1, Corollary 2.2]{ET1976}). 
Thus, by using $x^{k_j} \rightharpoonup x^\dag$ and $y^{k_j} \rightharpoonup y^\dag$ we obtain 
\begin{align*}
H(x^\dag, y^\dag) 
& = f(x^\dag) + g(y^\dag) \le \liminf_{j\to \infty} f(x^{k_j}) 
+ \liminf_{j\to \infty} g(y^{k_j}) \\
& \le \liminf_{j\to \infty} \left(f(x^{k_j}) + g(y^{k_j}) \right) \\
& = \liminf_{j \to \infty} H(x^{k_j}, y^{k_j}) = H_*.
\end{align*}
Since $(x^\dag, y^\dag)$ satisfies (\ref{admm-lc.43}), we also have $H(x^\dag, y^\dag) \ge H_*$. 
Therefore $H(x^\dag, y^\dag) = H_*$ and then it follows from (\ref{admm-lc.44}) and (\ref{admm-lc.43}) 
that 
$$
0 \le H(x, y) - H(x^\dag, y^\dag) + \l \la^\dag, A(x - x^\dag) + B (y - y^\dag)\r 
$$
for all $(x, y) \in \X \times \Y$. Using the definition of $H$ we can immediately see that 
$-A^* \la^\dag \in \p f(x^\dag)$ and $-B^* \la^\dag \in \p g(y^\dag)$. Therefore $u^\dag$ 
is a KKT point of (\ref{prob}). 

Let $u^*$ be another weak cluster point of $\{u^k\}$. Then there exists a subsequence 
$\{u^{l_j}\}$ of $\{u^k\}$ such that $u^{l_j} \rightharpoonup u^*$ as $j \to \infty$. 
Noting the identity 
\begin{align}\label{admm-conv}
2\l u^k, G(u^* - u^\dag)\r = \|u^k - u^\dag\|_G^2 - \|u^k - u^*\|_G^2 - \|u^\dag\|_G^2 + \|u^*\|_G^2.
\end{align}
Since both $u^*$ and $u^\dag$ are KKT points of (\ref{prob}) as shown above, 
it follows from Corollary \ref{cor2} that both $\{\|u^k - u^\dag\|_G^2\}$ and 
$\{\|u^k - u^*\|_G^2\}$ are monotonically decreasing and thus converge as 
$k \to \infty$. By taking $k = k_j$ and $k=l_j$ in (\ref{admm-conv}) respectively 
and letting $j \to \infty$ we can see that, for the both cases, the right hand 
side tends to the same limit. Therefore 
\begin{align*}
\l u^*, G(u^*-u^\dag)\r & = \lim_{j\to \infty} \l u^{l_j}, G(u^*-u^\dag)\r \\
& = \lim_{j\to \infty} \l u^{k_j}, G(u^*-u^\dag)\r \\
& = \l u^\dag, G(u^*-u^\dag)\r 
\end{align*}
which implies $\|u^*-u^\dag\|_G^2 =0$. 
\end{proof}

\begin{remark}\label{Rk.2}
{\rm 
Theorem \ref{thm2:ADMM} requires $\{u^k\}$ to be bounded. According to Corollary 
\ref{cor2}, $\{\|u^k\|_G^2\}$ is bounded which implies the boundedness of $\{\la^k\}$.
In the following we will provide sufficient conditions to guarantee the boundedness 
of $\{(x^k, y^k)\}$. 

\begin{enumerate}[leftmargin = 0.7cm]
\item[(i)] From (\ref{eq:cor.1})  
it follows that $\{\sigma_f \|x^k\|^2 + \sigma_g \|y^k\|^2 + \|u^k\|_G^2\}$ is bounded. 
By the definition of $G$, this in particular implies the boundedness of $\{\la^k\}$ and 
$\{B y^k\}$. Consequently, it follows from $\Delta \la^k = \rho(A x^k + B y^k - c)$
that $\{A x^k\}$ is bounded. Putting the above together we can conclude that both 
$\{(\sigma_f I + P + A^*A) x^k\}$ and $\{(\sigma_g I + Q + B^* B) y^k\}$ are bounded. 
Therefore, if both the bounded linear self-adjoint operators 
$$
\sigma_f I + P + A^*A \quad \mbox{and} \quad \sigma_g I + Q + B^*B
$$ 
are coercive, we can conclude the boundedness of $\{x^k\}$ and $\{y^k\}$. Here a 
linear operator $L: \V\to \H$ between two Hilbert spaces $\V$ and $\H$ is called 
coercive if $\|Lv\|\to \infty$ whenever $\|v\|\to \infty$. It is easy to see that $L$ 
is coercive if and only if there is a constant $c>0$ such that $c\|v\|\le \|Lv\|$ 
for all $v\in \V$.

\item[(ii)] If there exist $\beta>H_*$ and $\sigma>0$ such that the set 
$$
\{(x, y) \in \X\times \Y: H(x, y)\le \beta \mbox{ and } \|A x + B y -c\| \le \sigma\}
$$
is bounded, then $\{(x^k,y^k)\}$ is bounded. In fact, since $H(x^k, y^k)\to H_*$ and 
$A x^k + B y^k - c\to 0$ as shown in Theorem \ref{thm1}, the sequence $\{(x^k, y^k)\}$ is 
contained in the above set except for finite many terms. Thus $\{(x^k, y^k)\}$ 
is bounded.
\end{enumerate}
}
\end{remark}


\begin{remark} 
It is interesting to investigate under what conditions $\{u^k\}$ has a unique weak 
cluster point. According to Theorem \ref{thm2:ADMM}, for any two weak cluster points 
$u^* := (x^*, y^*, \la^*)$ and $u^\dag:= (x^\dag, y^\dag, \la^\dag)$ of $\{u^k\}$ 
there hold
\begin{align*}
& \|u^* - u^\dag\|_G^2 =0, \quad Ax^* + By^* = c, \quad -A^* \la^* \in \p f(x^*), 
\quad -B^* \la^* \in \p g(y^*), \\ 
& A x^\dag + B y^\dag = c, \quad -A^* \la^\dag \in \p f(x^\dag), \quad 
- B^* \la^\dag \in \p g(y^\dag).
\end{align*}
By using the definition of $G$ and the monotonicity of $\p f$ and $\p g$ we can deduce that 
\begin{align*} 
& \la^* = \la^\dag, \quad P(x^*-x^\dag) =0, \quad Q(y^*-y^\dag) = 0, \quad B(y^* - y^\dag) = 0, \\
& A(x^* - x^\dag) = 0, \quad \sigma_f \|x^* - x^\dag\|^2 =0, \quad \sigma_g \|y^* - y^\dag\|^2=0.
\end{align*}
Consequently 
$$
(\sigma_f I + P + A^*A) (x^*-x^\dag) =0 \quad \mbox{ and } \quad 
(\sigma_g I + Q + B^* B) (y^* - y^\dag) = 0.
$$
Therefore, if both $\sigma_f I + P + A^* A$ and $\sigma_g I + Q + B^* B$ are injective,
then $x^* = x^\dag$ and $y^* = y^\dag$ and hence $\{u^k\}$ has a unique weak cluster point, 
say $u^\dag$; consequently $u^k \rightharpoonup u^\dag$ as $k \to \infty$. 
\end{remark}

\begin{remark}
In \cite{Sun2019} the proximal ADMM (with relaxation) has been considered under 
the condition that 
\begin{align}\label{3.5}
P + \rho A^*A + \p f  \mbox{ and } Q + \rho B^*B + \p g \mbox{ are strongly maximal monotone}.
\end{align}
which requires both $(P + \rho A^*A + \p f)^{-1}$ and $(Q + \rho B^*B + \p g)^{-1}$ exist 
as single valued mappings and are Lipschitz continuous. It has been shown that the 
iterative sequence converges weakly to a KKT point which is its unique weak cluster point.  
The argument in \cite{Sun2019} used the facts that the KKT mapping $F(u)$, defined in 
(\ref{eq:F}) below, is maximal monotone and maximal monotone operators are closed 
under the weak-strong topology (\cite{AS2009b,BC2011}). Our argument is essentially 
based on Proposition 2.1, it is elementary and does not rely on any machinery from 
the maximal monotone operator theory.
\end{remark}

Based on Theorem \ref{thm2:ADMM}, we now devote to deriving convergence rates of 
the proximal ADMM (\ref{alg1}) under certain regularity conditions. To this end, we introduce 
the multifuncton $F: \X \times \Y \times \Z \rightrightarrows \X \times \Y \times \Z$ 
defined by 
\begin{align}\label{eq:F}
F(u):= \left(\begin{array}{ccc}
\p f(x) + A^* \la\\
\p g(y) + B^* \la\\
A x + By -c
\end{array}\right), \quad \forall u = (x, y, \la)\in \X \times \Y \times \Z.
\end{align}
Then $\bar u$ is a KKT point of (\ref{prob}) means $0 \in F(\bar u)$ or, equivalently, 
$\bar u\in F^{-1}(0)$, where $F^{-1}$ denotes the inverse multifunction of $F$. 
We will achieve our goal under certain bounded (H\"{o}lder) metric subregularity 
conditions of $F$. We need the following calculus lemma.

\begin{lemma}\label{lem5}
Let $\{\Delta_k\}$ be a sequence of nonnegative numbers satisfying
\begin{align}\label{5.8.6}
\Delta_k^\theta \le C (\Delta_{k-1}-\Delta_k)
\end{align}
for all $k\ge 1$, where $C>0$ and $\theta>1$ are constants. Then there is a constant $\tilde C>0$ such that
$$
\Delta_k \le \tilde C (1+ k)^{-\frac{1}{\theta-1}}
$$
for all $k\ge 0$.
\end{lemma}

\begin{proof}
Please refer to the proof of \cite[Theorem 2]{AB2009}. 
\end{proof}

\begin{theorem}\label{thm4:ADMM}
Let Assumption \ref{Ass1} and Assumption \ref{Ass2} hold. Consider the sequence 
$\{u^k:=(x^k, y^k, \la^k)\}$ generated by the proximal ADMM (\ref{alg1}). Assume 
$\{u^k\}$ is bounded and let $u^\dag :=(x^\dag, y^\dag, \la^\dag)$ be a 
weak cluster point of $\{u^k\}$. Let $R$ be a number such that 
$\|u^k - u^\dag\| \le R$ for all $k$ and assume that there exist $\kappa>0$ 
and $\a \in (0, 1]$ such that 
\begin{align}\label{HMS}
d(u, F^{-1}(0)) \le \kappa [d(0, F(u))]^\a, \quad \forall u \in B_R(u^\dag).
\end{align}
\begin{enumerate}[leftmargin = 0.8cm]
\item[\emph{(i)}] If $\a = 1$, then there exists a constant $0<q<1$ such that
\begin{align}\label{eq:lc}
\|u^{k+1} - u^\dag\|_G^2  + \|\Delta y^{k+1}\|_Q^2
\le q^2 \left( \|u^k - u^\dag\|_G^2 + \|\Delta y^k\|_Q^2\right)
\end{align}
for all $k \ge 0$ and consequently there exist $C>0$ and $0<q<1$ such that 
\begin{align}\label{eq:lc.1}
\begin{split}
\|u^k - u^\dag\|_G, \, \|\Delta u^k\|_G & \le C q^k, \\
\|A x^k + B y^k - c\| & \le C q^k, \\ 
|H(x^k, y^k) - H_*| & \le C q^k
\end{split}
\end{align}
for all $k \ge 0$. 

\item[\emph{(ii)}] If $\a \in (0, 1)$ then there is a constant $C$ such that
\begin{align}\label{eq:HCR}
\|u^k - u^\dag\|_G^2 + \|\Delta y^k\|_Q^2 \le C (k+1)^{-\frac{\a}{1-\a}}
\end{align}
and consequently
\begin{align}\label{eq:HCR.1}
\begin{split}
\|u^k - u^\dag\|_G, \, \|\Delta u^k \|_G & \le C (k+1)^{-\frac{1}{2(1-\a)}}, \\
\|A x^k + B y^k - c\| & \le C (k+1)^{-\frac{1}{2(1-\a)}}, \\
|H(x^k, y^k) - H_*| & \le C (k+1)^{- \frac{1}{2(1-\a)}}
\end{split}
\end{align}
for all $k \ge 0$.
\end{enumerate}
\end{theorem}

\begin{proof}
According to Theorem \ref{thm2:ADMM}, $u^\dag$ is a KKT point of (\ref{prob}). Therefore 
we may use Lemma \ref{prop9.20} with $\bar u = u^\dag$ to obtain 
\begin{align}\label{ADMM-LC.2}
\|u^{k+1} - u^\dag\|_G^2 + \|\Delta y^{k+1}\|_Q^2 
& \le \|u^k - u^\dag\|_G^2 + \|\Delta y^k\|_Q^2 - \|\Delta u^{k+1}\|_G^2 \nonumber \\
& = \|u^k- u^\dag\|_G^2 + \|\Delta y^k\|_Q^2 - \eta \|\Delta u^{k+1}\|_G^2 \nonumber \\
& \quad \, - (1-\eta) \|\Delta u^{k+1}\|_G^2,
\end{align}
where $\eta\in (0, 1)$ is any number. According to (\ref{11.11.1}),
$$
\begin{pmatrix}
\rho A^* B \Delta y^{k+1} - P \Delta x^{k+1}\\
- Q \Delta y^{k+1} \\
A x^{k+1} + B y^{k+1} - c
\end{pmatrix} \in F(u^{k+1}).
$$
Thus, by using $\Delta \la^{k+1} = \rho (A x^{k+1} + B y^{k+1}-c)$ we can obtain
\begin{align}\label{ADMM-LC.3}
d^2(0, F(u^{k+1})) 
& \le \|\rho A^* B \Delta y^{k+1} - P \Delta x^{k+1}\|^2 + \| -Q \Delta y^{k+1}\|^2 \nonumber \\
& \quad \, + \|A x^{k+1} + B y^{k+1}-c\|^2 \nonumber \\
& \le 2 \|P \Delta x^{k+1}\|^2 + 2 \rho^2 \|A\|^2 \|B \Delta y^{k+1}\|^2 \nonumber \\
& \quad \, + \|Q \Delta y^{k+1}\|^2 + \frac{1}{\rho^2} \|\Delta \la^{k+1}\|^2 \nonumber \\
& \le \gamma \|\Delta u^{k+1}\|_G^2,
\end{align}
where
$$
\gamma:= \max\left\{2 \|P\|, 2 \rho \|A\|^2, \|Q\|, \frac{1}{\rho}\right\}.
$$
Combining this with (\ref{ADMM-LC.2}) gives 
\begin{align*}
\|u^{k+1} - u^\dag\|_G^2 + \|\Delta y^{k+1} \|_Q^2 
& \le \|u^k - u^\dag\|_G^2 + \|\Delta y^k\|_Q^2 - \eta \|\Delta u^{k+1}\|_G^2 \\
& \quad \, - \frac{1-\eta}{\gamma} d^2(0, F(u^{k+1})). 
\end{align*}
Since $\|u^k - u^\dag\|\le R$ for all $k$ and $F$ satisfies (\ref{HMS}), one can see that  
$$
d(u^{k+1}, F^{-1}(0)) \le \kappa [d(0, F(u^{k+1}))]^\a, \quad \forall k\ge 0.
$$
Consequently
\begin{align*}
\|u^{k+1} - u^\dag\|_G^2 + \|\Delta y^{k+1}\|_Q^2 
& \le \|u^k - u^\dag\|_G^2 + \|\Delta y^k\|_Q^2 - \eta \|\Delta u^{k+1}\|_G^2 \\
& \quad \, - \frac{1-\eta}{\gamma \kappa^{2/\a}} [d(u^{k+1}, F^{-1}(0))]^{2/\a}. 
\end{align*}
For any $u =(x, y, \la) \in \X \times \Y \times \Z$ let
$$
d_G(u, F^{-1}(0)):= \inf_{\bar u\in F^{-1}(0)} \|u-\bar u\|_G
$$
which measures the ``distance" from $u$ to $F^{-1}(0)$ under the semi-norm $\|\cdot\|_G$. 
It is easy to see that
$$
d_G^2(u, F^{-1}(0)) \le \|G\| d^2(u, F^{-1}(0)),
$$
where $\|G\|$ denotes the norm of the operator $G$. Then we have 
\begin{align*}
\|u^{k+1} - u^\dag\|_G^2 + \|\Delta y^{k+1}\|_Q^2 
& \le \|u^k - u^\dag\|_G^2 + \|\Delta y^k\|_Q^2 - \eta \|\Delta u^{k+1}\|_G^2 \\
& \quad \, - \frac{1-\eta}{\gamma (\kappa^2 \|G\|)^{1/\a}} [d_G(u^{k+1}, F^{-1}(0))]^{2/\a}. 
\end{align*}
Now let $\bar u\in F^{-1}(0)$ be any point. Then 
\begin{align*}
\|u^{k+1} - u^\dag\|_G \le \|u^{k+1} - \bar u\|_G + \|u^\dag - \bar u\|_G.
\end{align*}
Since $u^\dag$ is a weak cluster point of $\{u^k\}$, there is a subsequence $\{u^{k_j}\}$ 
of $\{u^k\}$ such that $u^{k_j} \rightharpoonup u^\dag$. Thus   
\begin{align*}
\|u^\dag - \bar u\|_G^2 = \lim_{j \to \infty} \l u^{k_j} - \bar u, G(u^\dag - \bar u)\r 
\le \liminf_{j\to \infty} \|u^{k_j} - \bar u\|_G \|u^\dag - \bar u\|_G
\end{align*}
which implies $\|u^\dag - \bar u\|_G \le \liminf_{j\to \infty} \|u^{k_j} - \bar u\|_G$. 
From Corollary \ref{cor2} we know that $\{\|u^k - \bar u\|_G^2\}$ is monotonically decreasing. Thus
\begin{align*}
\|u^{k+1} - u^\dag\|_G
\le  \|u^{k+1} - \bar u\|_G + \liminf_{j\to \infty} \|u^{k_j} - \bar u\|_G \le 2 \|u^{k+1} - \bar u\|_G.
\end{align*}
Since $\bar u\in F^{-1}(0)$ is arbitrary, we thus have 
$$
\|u^{k+1}- u^\dag\|_G \le 2 d_G(u^{k+1}, F^{-1}(0)).
$$
Therefore 
\begin{align*}
\|u^{k+1} - u^\dag\|_G^2 + \|\Delta y^{k+1}\|_Q^2 
& \le \|u^k - u^\dag\|_G^2 + \|\Delta y^k\|_Q^2 - \eta \|\Delta u^{k+1}\|_G^2 \\
& \quad \, - \frac{1-\eta}{\gamma (4\kappa^2 \|G\|)^{1/\a}} \|u^{k+1} - u^\dag\|_G^{2/\a}. 
\end{align*}
By using the fact $\|\Delta u^k\|_G\to 0$ established in Proposition \ref{lem3}, we can 
find a constant $C>0$ such that 
$$
\|\Delta u^{k+1}\|_G^2 \ge C \|\Delta u^{k+1}\|_G^{2/\a}.
$$
Note that $\|\Delta u^{k+1}\|_G^2 \ge \|\Delta y^{k+1}\|_Q^2$. Thus 
\begin{align*}
\|u^{k+1} - u^\dag\|_G^2 + \|\Delta y^{k+1}\|_Q^2 
& \le \|u^k - u^\dag\|_G^2 + \|\Delta y^k\|_Q^2 - C \eta \|\Delta y^{k+1}\|_Q^{2/\a} \\
& \quad \, - \frac{1-\eta}{\gamma (4\kappa^2 \|G\|)^{1/\a}} \|u^{k+1} - u^\dag\|_G^{2/\a}. 
\end{align*}
Choose $\eta$ such that 
$$
\eta = \frac{1}{1 + C \gamma (4 \kappa^2 \|G\|)^{1/\a}}.
$$
Then 
\begin{align*}
& \|u^{k+1} - u^\dag\|_G^2 + \|\Delta y^{k+1}\|_Q^2 \\
& \le \|u^k - u^\dag\|_G^2 + \|\Delta y^k\|_Q^2 
- C \eta \left(\|\Delta y^{k+1}\|_Q^{2/\a} + \|u^{k+1} - u^\dag\|_G^{2/\a}\right). 
\end{align*}
Using the inequality $(a+b)^p \le 2^{p-1}(a^p + b^p)$ for $a,b\ge 0$ and $p\ge 1$, we then obtain 
\begin{align}\label{eq:lc.3}
& \|u^{k+1} - u^\dag\|_G^2 + \|\Delta y^{k+1}\|_Q^2 \nonumber\\
& \le \|u^k - u^\dag\|_G^2 + \|\Delta y^k\|_Q^2 
- 2^{1-1/\a} C \eta \left(\|u^{k+1} - u^\dag\|_G^2 + \|\Delta y^{k+1}\|_Q^2\right)^{1/\a}. 
\end{align}

(i) If $\a=1$, then we obtain the linear convergence   
$$
(1 + C\eta) \left(\|u^{k+1} - u^\dag\|_G^2 + \|\Delta y^{k+1}\|_Q^2\right)
\le \|u^k - u^\dag\|_G^2 + \|\Delta y^k\|_Q^2
$$
which is (\ref{eq:lc}) with $q = 1/(1 + C \eta)$. By using Lemma \ref{prop9.20} 
and (\ref{eq:lc}) we immediately obtain the first estimate in (\ref{eq:lc.1}). By using 
(\ref{eq:sbl.2}) and (\ref{eq:sbl}) we then obtain the last two estimates in (\ref{eq:lc.1}).  

(ii) If $\a\in (0, 1)$, we may use (\ref{eq:lc.3}) and Lemma \ref{lem5} to obtain (\ref{eq:HCR}).  
To derive the first estimate in (\ref{eq:HCR.1}), we may use Lemma \ref{prop9.20} to obtain 
$$
\sum_{j=l}^k \|\Delta u^j\|_G^2 \le \|u^l - u^\dag\|_G^2 + \|\Delta y^l\|_Q^2
$$
for all integers $1 \le l < k$. By using the monotonicity of $\{\|\Delta u^j\|_G^2\}$ shown in 
Lemma \ref{lem:mono} and the estimate (\ref{eq:HCR}) we have 
$$
(k-l+1) \|\Delta u^k\|_G^2 \le C (l+1)^{-\frac{\a}{1-\a}}. 
$$
Taking $l = [k/2]$, the largest integers $\le k/2$, gives 
$$
\|\Delta u^k\|_G^2 \le C (k + 1)^{-\frac{\a}{1-\a}-1} = C (k+1)^{-\frac{1}{1-\a}} 
$$
with a possibly different generic constant $C$. This shows the first estimate in (\ref{eq:HCR.1}).
Based on this, we can use (\ref{eq:sbl.2}) and (\ref{eq:sbl}) to obtain the last two estimates 
in (\ref{eq:HCR.1}). 
The proof is therefore complete. 
\end{proof}

\begin{remark}
Let us give some comments on the condition (\ref{HMS}). In finite dimensional 
Euclidean spaces, it has been proved in \cite{R1981} that for every polyhedral 
multifunction $\Psi: {\mathbb R}^m \rightrightarrows {\mathbb R}^n$ there is 
a constant $\kappa>0$ such that for any $y \in {\mathbb R}^n$ there is a number 
$\ep>0$ such that 
$$
d(x, \Psi^{-1}(y)) \le \kappa d(y, \Psi(x)), \quad \forall x \mbox{ satisfying } d(y, \Psi(x))<\ep.
$$
This result in particular implies the bounded metric subregularity of $\Psi$, 
i.e. for any $r>0$ and any $y \in {\mathbb R}^n$ there is a number $C>0$ such that 
$$
d(x, \Psi^{-1}(y)) \le C d(y, \Psi(x)), \quad \forall x \in B_r(0).
$$
Therefore, if $\p f$ and $\p g$ are polyhedral multifunctions, then the multifunction $F$ 
defined by (\ref{eq:F}) is also polyhedral and thus (\ref{HMS}) with $\a =1$ holds. 
The bounded metric subregularity of polyhedral multifunctions in arbitrary Banach spaces 
has been established in \cite{ZN2014}. 

On the other hand, if $\X$, $\Y$ and $\Z$ are finite dimensional Euclidean spaces, 
and if $f$ and $g$ are semi-algebraic convex functions, then the multifunction $F$ 
satisfies (\ref{HMS}) for some $\a \in (0, 1]$. Indeed, the semi-algebraicity of 
$f$ and $g$ implies that their subdifferentials $\p f$ and $\p g$ are semi-algebraic 
multifunctions with closed graph; consequently $F$ is semi-algebraic with closed graph.
According to \cite[Proposition 3.1]{LP2022}, $F$ is bounded H\"{o}lder metrically 
subregular at any point $(\bar u, \bar\xi)$ on its graph, i.e. for any $r>0$ there 
exist $\kappa>0$ and $\a\in (0,1]$ such that 
$$
d(u, F^{-1}(\bar \xi)) \le \kappa [d(\bar \xi, F(u))]^\a, \quad \forall u \in B_r(\bar u)
$$
which in particular implies (\ref{HMS}). 
\end{remark}

By inspecting the proof of Theorem \ref{thm4:ADMM}, it is easy to see that the same convergence rate 
results can be derived with the condition (\ref{HMS}) replaced by the weaker condition: there exist
$\kappa>0$ and $\a \in (0,1]$ such that 
\begin{align}\label{IBEBC}
d_G(u^k, F^{-1}(0)) \le \kappa \left\|\Delta u^k\right\|_G^\a, \quad \forall k \ge 1. 
\end{align}
Therefore we have the following result.

\begin{theorem}\label{thm5:ADMM}
Let Assumption \ref{Ass1} and Assumption \ref{Ass2} hold. Consider the sequence 
$\{u^k:=(x^k, y^k, \la^k)\}$ generated by the proximal ADMM (\ref{alg1}). Assume $\{u^k\}$ 
is bounded. If there exist $\kappa>0$ and $\a \in (0, 1]$ such that (\ref{IBEBC}) 
holds, then, for any weak cluster point $u^\dag$ of $\{u^k\}$, the same convergence 
rate results in Theorem \ref{thm4:ADMM} hold.  
\end{theorem}

\begin{remark}
Note that the condition (\ref{IBEBC}) is based on the iterative sequence itself. 
Therefore, it makes possible to check the condition by exploring not only the 
property of the multifunction $F$ but also the structure of the algorithm. The 
condition (\ref{IBEBC}) with $\a = 1$ has been introduced in \cite{LYZZ2018} as 
an iteration based error bound condition to study the linear convergence of 
the proximal ADMM (\ref{alg1}) with $Q =0$ in finite dimensions.  
\end{remark}

\begin{remark}
The condition (\ref{IBEBC}) is strongly motivated by the proof of Theorem \ref{thm4:ADMM}.
We would like to provide here an alternative motivation. Consider the proximal 
ADMM (\ref{alg1}). We can show that if $\|\Delta u^k\|_G =0$ then $u^k$ must be a KKT 
point of (\ref{prob}). Indeed, $\|\Delta u^k\|_G^2=0$ implies $P\Delta x^k =0$, 
${\widehat Q} \Delta y^k=0$ and $\Delta \la^k=0$. Since ${\widehat Q} = Q +\rho B^T B$ 
with $Q$ positive semi-definite and $\Delta \la^k = \rho (A x^k + B y^k-c)$, we also 
have $B \Delta y^k=0$, $Q\Delta y^k=0$ and $A x^k+B y^k-c=0$. Thus, it follows from 
(\ref{11.11.1}) that
$$
-A^* \la^k \in \p f(x^k), \quad -B^* \la^k \in \p g(y^k), \quad A x^k + B y^k =c
$$
which shows that $u^k=(x^k, y^k, \la^k)$ is a KKT point, i.e., $u^k \in F^{-1}(0)$. 
Therefore, it is natural to ask, if $\|\Delta u^k\|_G$ is small, can we guarantee 
$d_G(u^k, F^{-1}(0))$ to be small as well? This motivates us to propose 
a condition like 
$$
d_G(u^k, F^{-1}(0)) \le \varphi(\|\Delta u^k\|_G), \quad \forall k \ge 1
$$
for some function $\varphi: [0, \infty) \to [0, \infty)$ with $\varphi(0) =0$. The condition
(\ref{IBEBC}) corresponds to $\varphi(s) = \kappa s^\a$ for some $\kappa>0$ and $\a \in (0, 1]$. 
\end{remark}

In finite dimensional Euclidean spaces some linear convergence results on the 
proximal ADMM (\ref{alg1}) have been established in \cite{DY2016b} under various 
scenarios involving strong convexity of $f$ and/or $g$, Lipschitz continuity of 
$\nabla f$ and/or $\nabla g$, together with further conditions on $A$ and/or $B$, 
see \cite[Theorem 3.1 and Table 1]{DY2016b}. In the following theorem we will 
show that (\ref{IBEBC}) with $\a = 1$ holds under any one of these scenarios and 
thus the linear convergence in \cite[Theorem 3.1 and Theorem 3.4]{DY2016b} can be 
established by using Theorem \ref{thm5:ADMM}. Therefore, the linear convergence 
results based on the bounded metric subregularity of $F$ or the scenarios in \cite{DY2016b} 
can be treated in a unified manner. 

Actually our next theorem improves the results in \cite{DY2016b} by establishing the 
linear convergence of $\{u^k\}$ and $\{H(x^k, y^k)\}$ and relaxing the Lipschitz 
continuity of gradient(s) to the local Lipschitz continuity.
Furthermore, Our result is established in general Hilbert spaces. To formulate the scenarios 
from \cite{DY2016b} in this general setting, we need to replace the full row/column rank 
of matrices by the coercivity of linear operators. 
We also need the linear operator $M: \X\times \Y\to \Z$ defined by 
$$
M(x,y):=Ax+By, \quad \forall (x,y)\in \X\times \Y
$$
which is constructed from $A$ and $B$. It is easy to see that the adjoint of $M$ 
is $M^* z = (A^* z, B^* z)$ for any $z \in \Z$.

\begin{theorem}\label{thm2.11}
Let Assumption \ref{Ass1} and Assumption \ref{Ass2} hold. Let $\{u^k\}$ be the sequence 
generated by the proximal ADMM (\ref{alg1}). Then $\{u^k\}$ is bounded and there 
exists a constant $C>0$ such that
\begin{align}\label{ADMM.30}
d_G(u^k, F^{-1}(0)) \le C \|\Delta u^k\|_G
\end{align}
for all $k\ge 1$, provided any one of the following conditions holds:

\begin{enumerate}[leftmargin = 0.8cm]

\item[\emph{(i)}] $\sigma_g>0$, $A$ and $B^*$ are coercive, 
$g$ is differentiable and its gradient is Lipschitz continuous over bounded sets;

\item[\emph{(ii)}] $\sigma_f>0$, $\sigma_g>0$, $B^*$ is coercive, $g$ is differentiable 
and its gradient is Lipschitz continuous over bounded sets;

\item[\emph{(iii)}] $\la^0=0$, $\sigma_f>0$, $\sigma_g>0$, $M^*$ restricted on $\N(M^*)^\perp$ 
is coercive, both $f$ and $g$ are differentiable and their gradients are Lipschitz continuous 
over bounded sets;

\item[\emph{(iv)}] $\la^0=0$, $\sigma_g>0$, $A$ is coercive, $M^*$ restricted on $\N(M^*)^\perp$ 
is coercive, both $f$ and $g$ are differentiable and their gradients are Lipschitz continuous 
over bounded sets;
\end{enumerate}
where $\N(M^*)$ denotes the null space of $M^*$. 
Consequently, there exist $C>0$ and $0< q<1$ such that 
$$
\|u^k - u^\dag\| \le C q^k \quad \mbox{ and } \quad 
|H(x^k, y^k) - H_*| \le C q^{k}
$$
for all $k \ge 0$, where $u^\dag:=(x^\dag, y^\dag, \la^\dag)$ is a KKT point of (\ref{prob}).
\end{theorem}

\begin{proof}
We will only consider the scenario (i) since the proofs for other scenarios are similar. 
In the following we will use $C$ to denote a generic constant which may change from line 
to line but is independent of $k$. 

We first show the boundedness of $\{u^k\}$. According to Corollary \ref{cor2}, $\{\|u^k\|_G^2\}$ is 
bounded which implies the boundedness of $\{\la^k\}$. Since $\sigma_g>0$, it follows from 
(\ref{eq:cor.1}) that $\{y^k\}$ is bounded. Consequently, it follows from 
$\Delta \la^k = \rho(A x^k + B y^k -c)$ that $\{A x^k\}$ is bounded. Since $A$ is coercive,
$\{x^k\}$ must be bounded. 

Next we show (\ref{ADMM.30}). Let $u^\dag :=(x^\dag, y^\dag, \la^\dag)$ be a weak cluster 
point of $\{u^k\}$ whose existence is guaranteed by the boundedness of $\{u^k\}$. According to 
Theorem \ref{thm2:ADMM}, $u^\dag$ is a KKT point of (\ref{prob}). Let $(\xi, \eta, \tau) \in F(u^k)$ 
be any element. Then 
$$
\xi - A^* \la^k \in \p f(x^k), \quad \eta - B^* \la^k \in \p g(y^k), \quad 
\tau = A x^k + B y^k - c. 
$$
By using the monotonicity of $\p f$ and $\p g$ we have 
\begin{align}\label{ADMM.31}
& \sigma_f \|x^k - x^\dag\|^2 + \sigma_g \|y^k - y^\dag\|^2 \nonumber \\
& \le \l \xi - A^* \la^k + A^* \la^\dag, x^k - x^\dag\r + \l \eta - B^* \la^k + B^* \la^\dag, y^k - y^\dag\r \nonumber \\
& = \l \xi, x^k - x^\dag\r + \l \eta, y^k - y^\dag\r + \l \la^\dag - \la^k, A(x^k-x^\dag) + B(y^k - y^\dag)\r \nonumber \\
& = \l \xi, x^k - x^\dag\r + \l \eta, y^k - y^\dag\r + \l \la^\dag - \la^k, \tau\r.
\end{align}
Since $\sigma_g>0$, it follows from (\ref{ADMM.31}) and the Cauchy-Schwarz inequality that 
\begin{align}\label{ADMM.32}
\|y^k - y^\dag\|^2 \le C \left(\|\eta\|^2 + \|\xi\| \|x^k - x^\dag\| + \|\tau\| \|\la^k - \la^\dag\|\right).
\end{align}
Note that $A (x^k - x^\dag) = - B(y^k - y^\dag) + \frac{1}{\rho} \Delta\la^k$. Since $A$ is coercive,
we have 
\begin{align}\label{ADMM.34}
\|x^k - x^\dag\|^2 \le C \|A(x^k - x^\dag)\|^2 
\le C\left(\|y^k - y^\dag\|^2 + \|\Delta \la^k\|^2\right). 
\end{align}
By the differentiability of $g$ we have $-B^*\la^\dag = \nabla g(y^\dag)$ and 
$-B^* \la^k - Q \Delta y^k = \nabla g(y^k)$. Since $B^*$ is coercive and $\nabla g$ 
is Lipschitz continuous over bounded sets, we thus obtain 
\begin{align}\label{ADMM.33}
\|\la^k - \la^\dag\|^2 
& \le C \|B^*(\la^k - \la^\dag)\|^2 = \| Q\Delta y^k + \nabla g(y^k) - \nabla g(y^\dag)\|^2 \nonumber \\
& \le C \left(\|\Delta y^k\|_Q^2 + \|y^k - y^\dag\|^2\right).
\end{align}
Adding (\ref{ADMM.34}) and (\ref{ADMM.33}) and then using (\ref{ADMM.32}), it follows
\begin{align*}
\|x^k - x^\dag\|^2 + \|\la^k - \la^\dag\|^2 
& \le C \left(\|\eta\|^2 + \|\Delta u^k\|_G^2 + \|\xi\| \|x^k - x^\dag\| + \|\tau\| \|\la^k - \la^\dag\|\right) 
\end{align*}
which together with the Cauchy-Schwarz inequality then implies 
\begin{align}\label{ADMM.35}
\|x^k - x^\dag\|^2 + \|\la^k - \la^\dag\|^2 
\le C\left(\|\xi\|^2 + \|\eta\|^2 +\|\tau\|^2 + \|\Delta u^k\|_G^2\right).
\end{align}
Combining (\ref{ADMM.32}) and (\ref{ADMM.35}) we can obtain 
\begin{align*}
\|x^k - x^\dag\|^2 + \|y^k - y^\dag\|^2 + \|\la^k - \la^\dag\|^2 
\le C \left(\|\xi\|^2 + \|\eta\|^2 + \|\tau\|^2 + \|\Delta u^k\|_G^2\right). 
\end{align*}
Since $(\xi, \eta, \tau) \in F(u^k)$ is arbitrary, we therefore have
\begin{align*}
\|u^k - u^\dag\|^2 \le C \left([d(0, F(u^k))]^2 + \|\Delta u^k\|_G^2\right).
\end{align*}
With the help of (\ref{ADMM-LC.3}), we then obtain 
\begin{align}\label{admm.71}
\|u^k - u^\dag\|^2 \le C \|\Delta u^k\|_G^2.
\end{align}
Thus
\begin{align*}
[d_G(u^k, F^{-1}(0))]^2 \le C [d(u^k, F^{-1}(0))]^2 \le C\|u^k - u^\dag\|^2 
\le C \|\Delta u^k\|_G^2
\end{align*}
which shows (\ref{ADMM.30}). 

Because $\{u^k\}$ is bounded and (\ref{ADMM.30}) holds, we may use Theorem \ref{thm5:ADMM} 
to conclude the existence of a constant $q \in (0, 1)$ such that 
$$
\|\Delta u^k\|_G \le C q^k \quad \mbox{and} \quad |H(x^k, y^k) - H_*|\le C q^{k}.
$$ 
Finally we may use (\ref{admm.71}) to obtain $\|u^k-u^\dag\| \le C q^k$. 
\end{proof}

\begin{remark}
If $\Z$ is finite-dimensional, the coercivity of $M^*$ restricted on $\N(M^*)^\perp$ 
required in the scenarios (iii) and (iv) holds automatically. If it is not, then there 
exists a sequence $\{z^k\}\subset \N(M^*)^\perp\setminus\{0\}$ such that
$$
\|z^k\| \ge k \|M^* z^k\|, \quad k = 1, 2, \cdots. 
$$
By rescaling we may assume $\|z^k\|=1$ for all $k$. Since $\Z$ is finite-dimensional, by taking 
a subsequence if necessary, we may assume $z^k \to z$ for some $z \in \Z$. Clearly 
$z\in \N(M^*)^\perp$ and $\|z\|=1$. Note that $\|M^* z^k\|\le 1/k$ for all $k$, we have 
$\|M^* z\| = \lim_{k\to \infty} \|M^* z^k\|=0$ which means $z \in \N(M^*)$. Thus 
$z \in \N(M^*) \cap \N(M^*)^\perp = \{0\}$ which is a contradiction. 
\end{remark}

\section{\bf Proximal ADMM for linear inverse problems}\label{sect3}
\setcounter{equation}{0}

In this section we consider the method (\ref{PADMM2}) as a regularization method for 
solving (\ref{ip1.2}) and establish a convergence rate result under a benchmark source 
condition on the sought solution. Throughout this section we make the following 
assumptions on the operators $Q$, $L$, $A$, the constraint set $\C$ and the function $f$:

\begin{assumption}\label{ass:ADMM1}
\begin{enumerate}[leftmargin= 1cm]
\item[\emph{(i)}] $A: \X \to \H$ is a bounded linear operator, $Q: \X \to \X$ is a 
bounded linear positive semi-definite self-adjoint operator, and $\C \subset \X$ is 
a closed convex subset.

\item[\emph{(ii)}] $L$ is a densely defined, closed, linear operator from $\X$ to $\Y$ 
with domain $\emph{dom}(L)$.

\item[\emph{(iii)}] There is a constant $c_0>0$ such that
\begin{equation*}
\|A x\|^2 + \|L x\|^2 \ge c_0 \|x\|^2, \qquad \forall x\in \emph{dom}(L).
\end{equation*}

\item[\emph{(iv)}] $f: \mathcal{Y}\to (-\infty, \infty]$ is proper, lower semi-continuous, 
and strongly convex.
\end{enumerate}
\end{assumption}

This assumptions is standard in the literature on regularization methods and has been used in 
\cite{JJLW2016,JJLW2017}. Based on (iii), we can define the adjoint $L^*$ of $L$ which is 
also closed and densely defined; moreover, $z\in \mbox{dom}(L^*)$ if and only if
$\l L^* z, x\r =\l z, L x\r$ for all $x\in \mbox{dom}(L)$.  Under Assumption \ref{ass:ADMM1},
it has been shown in \cite{JJLW2016,JJLW2017} that the proximal ADMM (\ref{PADMM2}) is 
well-defined and if the exact data $b$ is consistent in the sense that there exists $\hat x\in \X$ 
such that 
$$
\hat x \in \mbox{dom}(L) \cap \C, \quad L \hat x \in \mbox{dom}(f) 
\quad \mbox{ and } \quad A \hat x = b,
$$
then the problem (\ref{ip1.2}) has a unique solution, denoted by $x^\dag$. Furthermore, 
there holds the following monotonicity result, see \cite[Lemma 2.3]{JJLW2017}; 
alternatively, it can also be derived from Lemma \ref{lem:mono}..

\begin{lemma}\label{lem1.1}
Let $\{z^k, y^k, x^k, \la^k, \mu^k, \nu^k\}$ be defined by the proximal ADMM 
(\ref{PADMM2}) with noisy data and let 
\begin{align}\label{eq:25}
E_k & := \frac{1}{2\rho_1} \|\Delta \la^k\|^2 + \frac{1}{2\rho_2} \|\Delta \mu^k\|^2 
+ \frac{1}{2\rho_3} \|\Delta \nu^k\|^2 \nonumber \\ 
& \quad \ + \frac{\rho_2}{2} \|\Delta y^k\|^2 + \frac{\rho_3}{2} \|\Delta x^k\|^2
+ \frac{1}{2} \|\Delta z^k\|_Q^2.
\end{align}
Then $\{E_k\}$ is monotonically decreasing with respect to $k$.
\end{lemma}

In the following we will always assume the exact data $b$ is consistent. 
We will derive a convergence rate of $x^k$ to the unique solution $x^\dag$ of (\ref{ip1.2})
under the source condition 
\begin{align}\label{ip.2}
\exists \mu^\dag \in \p f(L x^\dag) \cap \mbox{dom}(L^*) \mbox{ and }
\nu^\dag \in \p \iota_\C(x^\dag) \mbox{ such that }
L^* \mu^\dag + \nu^\dag \in \mbox{Ran}(A^*).
\end{align}
Note that when $L = I$ and $\C = \X$, (\ref{ip.2}) becomes the benchmark source condition 
$$
\p f(x^\dag) \cap \mbox{Ran}(A^*) \ne \emptyset
$$
which has been widely used to derive convergence rate for regularization methods,
see \cite{BO2004,FS2010,Jin2022,RS2006} for instance. We have the following 
convergence rate result. 

\begin{theorem}\label{thm:ip.admm}
Let Assumption \ref{ass:ADMM1} hold, let the exact data $b$ be consistent, and let 
the sequence $\{z^k, y^k, x^k, \la^k, \mu^k, \nu^k\}$ be defined by the proximal 
ADMM (\ref{PADMM2}) with noisy data $b^\d$ satisfying $\|b^\d - b\| \le \d$. Assume the 
unique solution $x^\dag$ of (\ref{ip1.2}) satisfies the source condition (\ref{ip.2}). 
Then for the integer $k_\d$ chosen by $k_\d \sim \d^{-1}$
there hold
$$
\|x^{k_\d} - x^\dag\| = O(\d^{1/4}), \quad \|y^{k_\d} - L x^\dag\| = O(\d^{1/4}) \quad 
\mbox{and} \quad \|z^{k_\d} - x^\dag\| = O(\d^{1/4})
$$
as $\d \to 0$. 
\end{theorem}

In order to prove this result, let us start from the formulation of the 
algorithm (\ref{PADMM2}) to derive some useful estimates. For simplicity 
of exposition, we set 
\begin{align*}
&\Delta x^{k+1}:= x^{k+1} - x^k, \quad \Delta y^{k+1}:= y^{k+1} - y^k, \quad 
\Delta z^{k+1}:= z^{k+1} - z^k, \\
&\Delta \la^{k+1}:= \la^{k+1} - \la^k, \quad \Delta \mu^{k+1}:= \mu^{k+1} - \mu^k, 
\quad \Delta \nu^{k+1}:= \nu^{k+1} - \nu^k.
\end{align*}
According to the definition of $z^{k+1}$, $y^{k+1}$ and $x^{k+1}$ in (\ref{PADMM2}), 
we have the optimality conditions 
\begin{align}
& 0 = A^* \la^k + \nu^k + \rho_1 A^* (A z^{k+1} - b^\d) + L^* (\mu^k + \rho_2(L z^{k+1} -y^k)) \nonumber \\
& \quad \quad + \rho_3 (z^{k+1} - x^k) + Q(z^{k+1} - z^k), \label{ip.3}\\
& 0 \in \p f(y^{k+1}) - \mu^k - \rho_2(L z^{k+1} - y^{k+1}), \label{ip.4}\\
& 0 \in \p\iota_C(x^{k+1}) -\nu^k - \rho_3 (z^{k+1} - x^{k+1}). \label{ip.5}
\end{align}
By using the last two equations in (\ref{PADMM2}), we have from (\ref{ip.4}) and 
(\ref{ip.5}) that 
\begin{align}
\mu^{k+1} \in \p f(y^{k+1}) \quad \mbox{and} \quad 
\nu^{k+1} \in \p\iota_{\C}(x^{k+1}). \label{ip.7}
\end{align}
Letting $y^\dag = L x^\dag$. From the strong convexity of $f$, the convexity of $\iota_C$,
and (\ref{ip.7}) it follows that  
\begin{align}\label{ip.8}
\sigma_f \|y^{k+1} - y^\dag\|^2 
& \le f(y^\dag) - f(y^{k+1}) - \l \mu^{k+1}, y^\dag - y^{k+1}\r \nonumber \\
& \quad \, + \l \nu^{k+1}, x^{k+1} - x^\dag\r. 
\end{align}
where $\sigma_f$ denotes the modulus of convexity of $f$; we have $\sigma_f>0$ as $f$ is 
strongly convex. By taking the inner product of (\ref{ip.3}) with $z^{k+1} - x^\dag$ we have 
\begin{align*}
0 
& = \l \la^k + \rho_1 (A z^{k+1} - b^\d), A (z^{k+1} - x^\dag)\r \\
& \quad \, + \l \mu^k + \rho_2(L z^{k+1} - y^k), L(z^{k+1} - x^\dag)\r \nonumber \\
& \quad \, + \l \nu^k + \rho_3 (z^{k+1} - x^k), z^{k+1} - x^\dag\r \\
& \quad \, + \l Q(z^{k+1} - z^k), z^{k+1} - x^\dag\r.
\end{align*}
Therefore we may use the definition of $\la^{k+1}, \mu^{k+1}, \nu^{k+1}$ in (\ref{PADMM2}) and the 
fact $A x^\dag = b$ to further obtain
\begin{align}\label{ip.9}
0 & = \l \la^{k+1}, A z^{k+1} - b\r + \l \mu^{k+1} 
+ \rho_2 \Delta y^{k+1}, L z^{k+1} - y^\dag\r \nonumber \\
& \quad \, + \l \nu^{k+1} + \rho_3 \Delta x^{k+1}, z^{k+1} - x^\dag\r \nonumber \\
& \quad \, + \l Q\Delta z^{k+1}, z^{k+1} - x^\dag\r. 
\end{align}
Subtracting (\ref{ip.8}) by (\ref{ip.9}) gives 
\begin{align*}
\sigma_f \|y^{k+1} - y^\dag\|^2 
& \le f(y^\dag) - f(y^{k+1}) - \l \la^{k+1}, A z^{k+1} - b\r + \l \mu^{k+1}, y^{k+1} - L z^{k+1}\r \\
& \quad \, - \rho_2 \l \Delta y^{k+1}, L z^{k+1} - y^\dag\r + \l \nu^{k+1}, x^{k+1} - z^{k+1}\r \\
& \quad \, -\rho_3 \l \Delta x^{k+1}, z^{k+1} - x^\dag\r - \l Q\Delta z^{k+1}, z^{k+1} - x^\dag\r.
\end{align*}
Note that under the source condition (\ref{ip.2}), there exist $\mu^\dag$, $\nu^\dag$ 
and $\la^\dag$ such that 
\begin{align}\label{ip.9.5}
\mu^\dag \in \p f(y^\dag), \quad \nu^\dag \in \p \iota_\C(x^\dag) \quad \mbox{ and } 
\quad L^* \mu^\dag + \nu^\dag + A^* \la^\dag = 0. 
\end{align}
Thus, it follows from the above equation and the last two equations in (\ref{PADMM2}) that 
\begin{align*}
& \sigma_f \|y^{k+1} - y^\dag\|^2 \\
& \le f(y^\dag) - f(y^{k+1}) - \l \la^\dag, A z^{k+1} - b\r 
- \l \mu^\dag,  L z^{k+1} - y^{k+1}\r - \l \nu^\dag, z^{k+1} - x^{k+1}\r\\
& \quad \ - \l \la^{k+1} - \la^\dag, A z^{k+1} - b^\d + b^\d - b\r \\
& \quad \ -\frac{1}{\rho_2} \l \mu^{k+1} - \mu^\dag, \Delta\mu^{k+1} \r 
 - \rho_2 \l \Delta y^{k+1}, L z^{k+1} - y^\dag\r \\
& \quad \ - \frac{1}{\rho_3} \l \nu^{k+1} - \nu^\dag, \Delta\nu^{k+1}\r 
- \rho_3 \l \Delta x^{k+1}, z^{k+1} - x^\dag\r \\
& \quad \, - \l Q\Delta z^{k+1}, z^{k+1} - x^\dag\r. 
\end{align*}
By using (\ref{ip.9.5}), $b = A x^\dag$ and the convexity of $f$, we can see that 
\begin{align*}
& f(y^\dag) - f(y^{k+1}) - \l \la^\dag, A z^{k+1} - b\r 
- \l \mu^\dag,  L z^{k+1} - y^{k+1}\r - \l \nu^\dag, z^{k+1} - x^{k+1}\r\\
& = f(y^\dag) - f(y^{k+1}) + \l \la^\dag, b\r + \l \mu^\dag, y^{k+1}\r 
+ \l \nu^\dag,  x^{k+1}\r\\
& = f(y^\dag) - f(y^{k+1}) + \l A^* \la^\dag, x^\dag\r + \l \mu^\dag, y^{k+1}\r 
+ \l \nu^\dag,  x^{k+1}\r\\
& = f(y^\dag) - f(y^{k+1}) - \l L^* \mu^\dag, x^\dag\r + \l \mu^\dag, y^{k+1}\r 
+ \l \nu^\dag,  x^{k+1} - x^\dag \r\\
& = f(y^\dag) - f(y^{k+1}) + \l \mu^\dag, y^{k+1} - y^\dag\r 
+ \l \nu^\dag,  x^{k+1} - x^\dag \r \le 0.
\end{align*}
Consequently, by using the fourth equation in (\ref{PADMM2}), we have
\begin{align*}
\sigma_f \|y^{k+1} - y^\dag\|^2 
& \le - \l \la^{k+1} - \la^\dag, b^\d - b\r 
- \frac{1}{\rho_1} \l \la^{k+1} - \la^\dag, \Delta \la^{k+1}\r \\
& \quad \ - \frac{1}{\rho_2} \l \mu^{k+1} - \mu^\dag, \Delta\mu^{k+1}\r  
- \frac{1}{\rho_3} \l \nu^{k+1} - \nu^\dag, \Delta\nu^{k+1}\r \\
& \quad \ - \rho_2 \l \Delta y^{k+1}, y^{k+1}-y^\dag + L z^{k+1} - y^{k+1}\r \\
& \quad \ - \rho_3 \l \Delta x^{k+1}, x^{k+1} - x^\dag + z^{k+1} - x^{k+1}\r \\
& \quad \ - \l Q\Delta z^{k+1}, z^{k+1} - x^\dag\r. 
\end{align*}
By using the polarization identity and the last two equations in (\ref{PADMM2}) we 
further have 
\begin{align*}
\sigma_f \|y^{k+1} - y^\dag\|^2 
& \le - \l \la^{k+1} - \la^\dag, b^\d - b\r \\
& \quad \ + \frac{1}{2\rho_1} \left(\|\la^k -\la^\dag\|^2 - \|\la^{k+1} - \la^\dag\|^2 
- \|\Delta \la^{k+1}\|^2\right) \\
& \quad \ + \frac{1}{2\rho_2} \left(\|\mu^k - \mu^\dag\|^2 - \|\mu^{k+1} - \mu^\dag\|^2 
- \|\Delta\mu^{k+1}\|^2\right) \\
& \quad \ + \frac{1}{2\rho_3} \left(\|\nu^k - \nu^\dag\|^2 - \|\nu^{k+1} - \nu^\dag\|^2 
- \|\Delta \nu^{k+1}\|^2 \right) \\
& \quad \ + \frac{1}{2} \left(\|z^k - z^\dag\|_Q^2 - \|z^{k+1} - z^\dag\|_Q^2 
- \|\Delta z^{k+1}\|_Q^2\right) \\
& \quad \ + \frac{\rho_2}{2} \left(\|y^k - y^\dag\|^2 - \|y^{k+1} - y^\dag\|^2 
- \|\Delta y^{k+1}\|^2\right) \\
& \quad \ + \frac{\rho_3}{2} \left(\|x^k - x^\dag\|^2 - \|x^{k+1} - x^\dag\|^2 
- \|\Delta x^{k+1}\|^2\right) \\
& \quad \ - \l \Delta y^{k+1}, \Delta \mu^{k+1} \r 
- \l \Delta x^{k+1}, \Delta \nu^{k+1} \r.
\end{align*}
Let 
\begin{align*}
\Phi_k & := \frac{1}{2\rho_1} \|\la^k -\la^\dag\|^2 
+ \frac{1}{2\rho_2} \|\mu^k - \mu^\dag\|^2 
+ \frac{1}{2\rho_3} \|\nu^k - \nu^\dag\|^2 \\
& \quad \ + \frac{1}{2} \|z^k - x^\dag\|_Q^2 
+ \frac{\rho_2}{2} \|y^k - y^\dag\|^2
+ \frac{\rho_3}{2} \|x^k - x^\dag\|^2. 
\end{align*}
Then 
\begin{align}\label{ip.10}
\sigma_f \|y^{k+1} - y^\dag\|^2 
& \le \Phi_k - \Phi_{k+1} - \l \la^{k+1} - \la^\dag, b^\d - b\r - E_{k+1} \nonumber \\
& \quad \ - \l \Delta y^{k+1}, \Delta \mu^{k+1}\r 
- \l \Delta x^{k+1}, \Delta \nu^{k+1}\r,
\end{align}
where $E_k$ is defined by (\ref{eq:25}). 

\begin{lemma}\label{lem12}
For all $k = 0, 1, \cdots$ there hold 
\begin{align}\label{ip.11}
\sigma_f \|y^{k+1} - y^\dag\|^2 
\le \Phi_k - \Phi_{k+1} - \l \la^{k+1} - \la^\dag, b^\d - b\r - E_{k+1},
\end{align}
\begin{align}\label{ip.12}
E_{k+1} \le \Phi_k - \Phi_{k+1} + \sqrt{2\rho_1 \Phi_{k+1}} \d
\end{align}
and 
\begin{align}\label{ip.13}
\Phi_{k+1} \le \Phi_0 + \left(\sum_{j=1}^{k+1} \sqrt{2\rho_1 \Phi_j}\right) \d. 
\end{align}
\end{lemma}

\begin{proof}
By using (\ref{ip.7}) and the monotonicity of the subdifferentials $\p f$ 
and $\p \iota_{\C}$ we have 
$$
0 \le \sigma_f \|\Delta y^{k+1}\|^2 \le \l \Delta \mu^{k+1}, \Delta y^{k+1}\r 
+ \l \Delta \nu^{k+1}, \Delta x^{k+1}\r.
$$
This together with (\ref{ip.10}) implies (\ref{ip.11}). From (\ref{ip.11}) it follows 
immediately that
\begin{align*}
E_{k+1} &\le \Phi_k - \Phi_{k+1} - \l \la^{k+1} - \la^\dag, b^\d - b\r \nonumber\\
& \le \Phi_k - \Phi_{k+1} + \|\la^{k+1} - \la^\dag\| \d \nonumber \\
& \le \Phi_k - \Phi_{k+1} + \sqrt{2\rho_1 \Phi_{k+1}} \d
\end{align*}
which shows (\ref{ip.12}). By the non-negativity of $E_{k+1}$ we then obtain from (\ref{ip.12}) that
$$
\Phi_{k+1} \le \Phi_k + \sqrt{2 \rho_1 \Phi_{k+1}} \d, \quad \forall k \ge 0  
$$
which clearly implies (\ref{ip.13}). 
\end{proof}

In order to derive the estimate on $\Phi_k$ from (\ref{ip.13}), we need the following elementary result.

\begin{lemma}\label{lem4}
Let $\{a_k\}$ and $\{b_k\}$ be two sequences of nonnegative numbers such that
$$
a_k^2 \le b_k^2 + c \sum_{j=1}^{k} a_j, \quad k=0, 1, \cdots,
$$
where $c \ge 0$ is a constant. If $\{b_k\}$ is non-decreasing, then
$$
a_k \le b_k + c k, \quad k=0, 1, \cdots.
$$
\end{lemma}

\begin{proof}
We show the result by induction on $k$. The result is trivial for $k =0$ since the given condition 
with $k=0$ gives $a_0 \le b_0$. Assume that the result is valid for all $0\le k \le l$ for 
some $l\ge 0$. We show it is also true for $k = l+1$. If $a_{l+1} \le \max\{a_0, \cdots, a_l\}$, 
then $a_{l+1} \le a_j$ for some $0\le j\le l$. Thus, by the induction hypothesis and the 
monotonicity of $\{b_k\}$ we have
$$
a_{l+1} \le a_j \le b_j + c j \le b_{l+1} + c (l+1).
$$
If $a_{l+1} > \max\{a_0, \cdots, a_l\}$, then
\begin{align*}
a_{l+1}^2 \le b_{l+1}^2 + c \sum_{j=1}^{l+1} a_j \le b_{l+1}^2 + c(l+1) a_{l+1}
\end{align*}
which implies that
\begin{align*}
\left(a_{l+1} - \frac{1}{2} c (l+1)\right)^2
& = a_{l+1}^2 - c (l+1) a_{l+1} + \frac{1}{4} c^2 (l+1)^2\\
& \le b_{l+1}^2 + \frac{1}{4} c^2 (l+1)^2 \\
& \le \left(b_{l+1} + \frac{1}{2} c (l+1)\right)^2.
\end{align*}
Taking square roots shows $a_{l+1} \le b_{l+1} + c (l+1)$ again. 
\end{proof}

\begin{lemma}\label{lem11}
There hold
\begin{align}\label{ip.14}
\Phi_k^{1/2} \le \Phi_0^{1/2} + \sqrt{2 \rho_1} k \d, \quad \forall k \ge 0 
\end{align}
and 
\begin{align}\label{ip.15}
E_k \le \frac{2\Phi_0}{k} + \frac{5}{2} \rho_1 k \d^2, \quad \forall k \ge 1.
\end{align}
\end{lemma}

\begin{proof} 
Based on (\ref{ip.13}), we may use Lemma \ref{lem4} with $a_k = \Phi_k^{1/2}$, $b_k = \Phi_0^{1/2}$ 
and $c = (2\rho_2)^{1/2} \d$ to obtain (\ref{ip.14}) directly. 
Next, by using the monotonicity of $\{E_k\}$, (\ref{ip.12}) and (\ref{ip.14}) we have 
\begin{align*}
k E_k &\le \sum_{j=1}^k E_j 
\le \sum_{j=1}^k \left(\Phi_{j-1} - \Phi_j + \sqrt{2\rho_1 \Phi_j}\d \right) \\
& \le \Phi_0 - \Phi_k + \sum_{j=1}^k \sqrt{2\rho_1 \Phi_j} \d \\
& \le \Phi_0 + \sum_{j=1}^k \sqrt{2\rho_1} \left(\sqrt{\Phi_0} + \sqrt{2\rho_1} j \d\right) \d \\
& =  \Phi_0 + \sqrt{2\rho_1 \Phi_0} k \d + \rho_1 k(k+1) \d^2 \\
& \le  2 \Phi_0 + \frac{5}{2} \rho_1 k^2 \d^2
\end{align*}
which shows (\ref{ip.15}). 
\end{proof}

Now we are ready to complete the proof of Theorem \ref{thm:ip.admm}.

\begin{proof}[Proof of Theorem \ref{thm:ip.admm}]
Let $k_\d$ be an integer such that $k_\d \sim \d^{-1}$. From (\ref{ip.14}) and (\ref{ip.15}) 
in Lemma \ref{lem11} it follows that  
\begin{align}\label{ip.16}
E_{k_\d} \le C_0 \d \quad \mbox{and} \quad \Phi_k \le C_1 \mbox{ for all } k\le k_\d,
\end{align}
where $C_0$ and $C_1$ are constants independent of $k$ and $\d$. In order to use (\ref{ip.11}) 
in Lemma \ref{lem12} to estimate $\|y^{k_\d} - y^\dag\|$, we first consider $\Phi_k - \Phi_{k+1}$
for all $k\ge 0$. By using the definition of $\Phi_k$ and the inequality 
$\|u\|^2 - \|v\|^2 \le (\|u\| + \|v\|) \|u - v\|$, we have for $k \ge 0$ that  
\begin{align*}
\Phi_{k} - \Phi_{k+1}  
& \le \frac{1}{2\rho_1} \left(\|\la^{k} - \la^\dag\| + \|\la^{k+1} - \la^\dag\|\right) \|\Delta \la^{k+1}\| \\
& \quad \ + \frac{1}{2\rho_2} \left(\|\mu^{k} - \mu^\dag\| + \|\mu^{k+1} - \mu^\dag\|\right)\|\Delta \mu^{k+1}\| \\ 
& \quad \ + \frac{1}{2\rho_3} \left(\|\nu^{k} - \nu^\dag\| + \|\nu^{k+1} - \nu^\dag\|\right) \|\Delta \nu^{k+1}\| \\
& \quad \ + \frac{1}{2}\left( \|z^k - x^\dag\|_Q + \|z^{k+1}-x^\dag\|_Q\right)\|\Delta z^{k+1}\|_Q \\ 
& \quad \ + \frac{\rho_2}{2} \left(\|y^k - y^\dag\| + \|y^{k+1} - y^\dag\|\right) \|\Delta y^{k+1}\| \\
& \quad \ + \frac{\rho_3}{2} \left(\|x^k - x^\dag\| + \|x^{k+1} - x^\dag\|\right) \|\Delta x^{k+1}\|.
\end{align*}
By virtue of the Cauchy-Schwarz inequality and the inequality $(a+b)^2 \le 2(a^2 + b^2)$ for any numbers 
$a, b \in {\mathbb R}$ we can further obtain 
\begin{align*}
\Phi_k - \Phi_{k+1} \le \sqrt{2(\Phi_k + \Phi_{k+1}) E_{k+1}}, \quad \forall k\ge 0. 
\end{align*}
This together with (\ref{ip.16}) in particular implies 
$$
\Phi_{k_\d-1} - \Phi_{k_\d} \le \sqrt{4 C_0 C_1 \d}.  
$$
Therefore, it follows from (\ref{ip.11}) that  
\begin{align*}
\sigma_f \|y^{k_\d}-y^\dag\|^2 
& \le \Phi_{k_\d-1} - \Phi_{k_\d} + \|\la^{k_\d}-\la^\dag\| \d \\
& \le \sqrt{4 C_0C_1 \d} + \sqrt{2\rho_1 \Phi_{k_\d}} \d \\
& \le \sqrt{4 C_0 C_1 \d} + \sqrt{2\rho_1 C_1} \d.
\end{align*}
Thus 
\begin{align*}
\|y^{k_\d} - y^\dag\|^2 \le C_2 \d^{1/2},
\end{align*}
where $C_2$ is a constant independent of $\d$ and $k$. By using the estimate 
$E_{k_\d} \le C_0 \d$ in (\ref{ip.16}), the definition of $E_{k_\d}$, and the last 
three equations in (\ref{PADMM2}), we can see that 
\begin{align*}
& \|A z^{k_\d}- b^\d\|^2 = \frac{1}{\rho_1^2} \|\Delta \la^{k_\d}\|^2 
\le \frac{2}{\rho_1} E_{k_\d} \le \frac{2 C_0}{\rho_1} \d,\\
& \|L z^{k_\d} - y^{k_\d}\|^2 = \frac{1}{\rho_2^2} \|\Delta \mu^{k_\d}\|^2 
\le \frac{2}{\rho_2} E_{k_\d} \le \frac{2C_0}{\rho_2} \d,\\
& \|z^{k_\d} - x^{k_\d}\|^2 = \frac{1}{\rho_3^2} \|\Delta \nu^{k_\d}\|^2 
\le \frac{2}{\rho_3} E_{k_\d} \le \frac{2 C_0}{\rho_3} \d. 
\end{align*}
Therefore 
\begin{align*}
&\|L(z^{k_\d}-x^\dag)\|^2 \le 2\left(\|L z^{k_\d} - y^{k_\d}\|^2 + \|y^{k_\d} - y^\dag\|^2\right) 
\le \frac{4 C_0}{\rho_2} \d + 2 C_2 \d^{1/2}, \\
& \|A(z^{k_\d}- x^\dag)\|^2 \le 2\left(\|A z^{k_\d} - b^\d\|^2 + \|b^\d - b\|^2\right) 
\le 2 \left(\frac{2 C_0}{\rho_1} + 1\right) \d.
\end{align*}
By virtue of (iii) in Assumption \ref{ass:ADMM1} on $A$ and $L$ we thus obtain 
\begin{align*}
c_0 \|z^{k_\d}-x^\dag\|^2 
& \le \|A(z^{k_\d} - x^\dag)\|^2 + \|L(z^{k_\d}- x^\dag)\|^2 \\
& \le 2 \left(\frac{2 C_0}{\rho_1} + \frac{2 C_0}{\rho_1} + 1\right) \d + 2 C_2 \d^{1/2}. 
\end{align*}
This means there is a constant $C_3$ independent of $\d$ and $k$ such that 
\begin{align*}
\|z^{k_\d} - x^\dag\|^2 \le C_3 \d^{1/2}. 
\end{align*}
Finally we obtain 
$$
\|x^{k_\d} - x^\dag\|^2 \le 2\left(\|x^{k_\d}- z^{k_\d}\|^2 + \|z^{k_\d} - x^\dag\|^2\right) 
\le \frac{4 C_0}{\rho_3} \d + 2 C_3 \d^{1/2}. 
$$
The proof is thus complete. 
\end{proof}

\begin{remark}
Under the benchmark source condition (\ref{ip.2}), we have obtained in Theorem 
\ref{thm:ip.admm} the convergence rate $O(\d^{1/4})$ for the proximal ADMM (\ref{PADMM2}). 
This rate is not order optimal. It is not yet clear if the order optimal rate $O(\d^{1/2})$ 
can be achieved. 
\end{remark}

\begin{remark}
When using the proximal ADMM to solve (\ref{ip1.2}) with $L = I$, i.e.
\begin{align}\label{ip}
\min\left\{f(x): A x = b \mbox{ and } x \in \C\right\}, 
\end{align}
it is not necessary to introduce the $y$-variable as is done in (\ref{PADMM2})
and thus (\ref{PADMM2}) can be simplified to the scheme
\begin{align}\label{PADMM2.5}
\begin{split}
& z^{k+1} = \arg\min_{z\in \X}
\left\{{\mathscr L}_{\rho_1, \rho_2}(z, x^k, \la^k, \nu^k) + \frac{1}{2} \|z-z^k\|_Q^2\right\},\\
& x^{k+1} = \arg\min_{x\in \X} 
\left\{{\mathscr L}_{\rho_1, \rho_2}(z^{k+1}, x, \la^k, \nu^k)\right\}, \\
& \la^{k+1} = \la^k + \rho_1 (A z^{k+1} - b^\d), \\
& \nu^{k+1} = \nu^k + \rho_2 (z^{k+1} - x^{k+1}),
\end{split}
\end{align}
where
$$
{\mathscr L}_{\rho_1, \rho_2}(z, x, \la, \nu) 
:= f(z) + \iota_\C(x) + \l \la, A z- b^\d\r + \l \nu, z-x\r 
+ \frac{\rho_1}{2}\|A z- b^\d\|^2 + \frac{\rho_2}{2} \|z-x\|^2.
$$
The source condition (\ref{ip1.2}) reduces to the form 
\begin{align}\label{sc}
\exists \mu^\dag \in \p f(x^\dag) \mbox{ and } \nu^\dag \in \p \iota_\C(x^\dag) 
\mbox{ such that } \mu^\dag + \nu^\dag \in \mbox{Ran}(A^*).
\end{align}
If the unique solution $x^\dag$ of (\ref{ip}) satisfies the source condition (\ref{sc}), 
one may follow the proof of Theorem \ref{thm:ip.admm} with minor modification
to deduce for the method (\ref{PADMM2.5}) that 
$$
\|x^{k_\d} - x^\dag\| = O(\d^{1/4}) \quad \mbox{ and } \quad 
\|z^{k_\d} - x^\dag\| = O(\d^{1/4})
$$
whenever the integer $k_\d$ is chosen such that $k_\d \sim \d^{-1}$. 
\end{remark}

We conclude this section by presenting a numerical result to illustrate the 
semi-convergence property of the proximal ADMM and the convergence rate. We consider 
finding a solution of (\ref{ip1.1}) with minimal norm. This is equivalent to solving 
(\ref{ip}) with $f(y) = \frac{1}{2} \|y\|^2$. With a noisy data $b^\d$ satisfying 
$\|b^\d - b\| \le \d$, the corresponding proximal ADMM (\ref{PADMM2.5}) takes the form 
\begin{align}\label{PADMM2.6}
\begin{split}
z^{k+1} &= \left((1 + \rho_2) I + Q  + \rho_1 A^*A\right)^{-1} 
\left(\rho_1 A^* b^\d + \rho_2 x^k + Q z^k - A^* \la^k - \nu^k\right), \\
x^{k+1} &= P_\C\left(z^{k+1} + \nu^k/\rho_2\right), \\
\la^{k+1} & = \la^k + \rho_1(A z^{k+1} - b^\d), \\
\nu^{k+1} & = \nu^k + \rho_2(z^{k+1} - x^{k+1}), 
\end{split}
\end{align}
where $P_\C$ denotes the orthogonal projection of $\X$ onto $\C$. The source condition (\ref{sc}) 
now takes the form 
\begin{align}\label{sc.2}
\exists \nu^\dag \in \p \iota_\C (x^\dag) \mbox{ such that } x^\dag + \nu^\dag \in \mbox{Ran}(A^*)
\end{align}
which is equivalent to the projected source condition $x^\dag \in P_\C(\mbox{Ran}(A^*))$. 

\begin{example}
In our numerical simulation we consider the first kind integral equation
\begin{align}\label{IE}
(A x)(t) := \int_0^1 \kappa(s, t) x(s) ds = b(t), \quad t \in [0,1]
\end{align}
on $L^2[0,1]$, where the kernel $\kappa$ is continuous on $[0,1]\times [0,1]$. Such 
equations arise naturally in many linear ill-posed inverse problems, see 
\cite{EHN1996,G1984}. Clearly $A$ is a compact linear operator from $L^2[0,1]$ 
to itself. We will use 
$$
\kappa(s, t) = d \left(d^2 + (s-t)^2\right)^{-3/2}
$$
with $d = 0.1$. The corresponding equation is a 1-D model problem in gravity 
surveying. Assume the equation (\ref{IE}) has a nonnegative solution. We will 
employ the method (\ref{PADMM2.6}) to determine the unique nonnegative solution 
of (\ref{IE}) with minimal norm in case the data is corrupted by noise.  
Here $\C:=\{x \in L^2[0,1]: x\ge 0 \mbox{ a.e.}\}$ and thus $P_\C(x) = \max\{x, 0\}$. 

\begin{figure}[ht]
        \centering
            \includegraphics[width=0.31\textwidth]{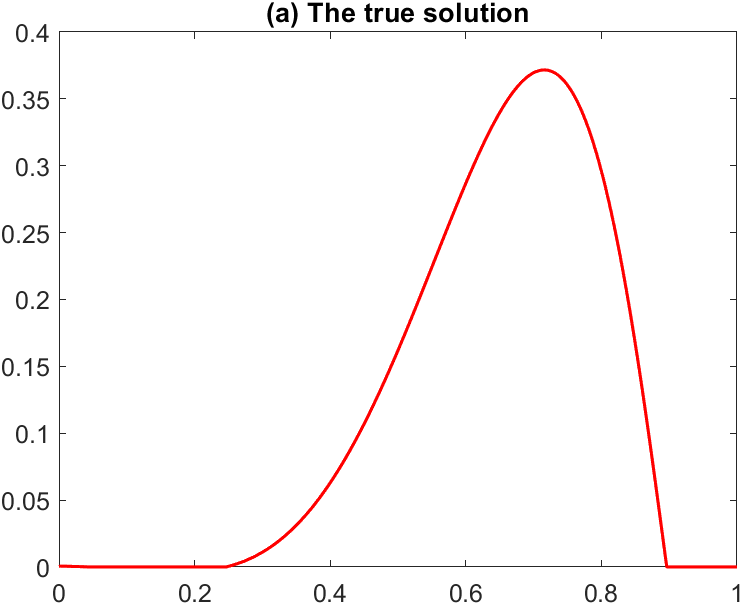}
            \includegraphics[width=0.31\textwidth]{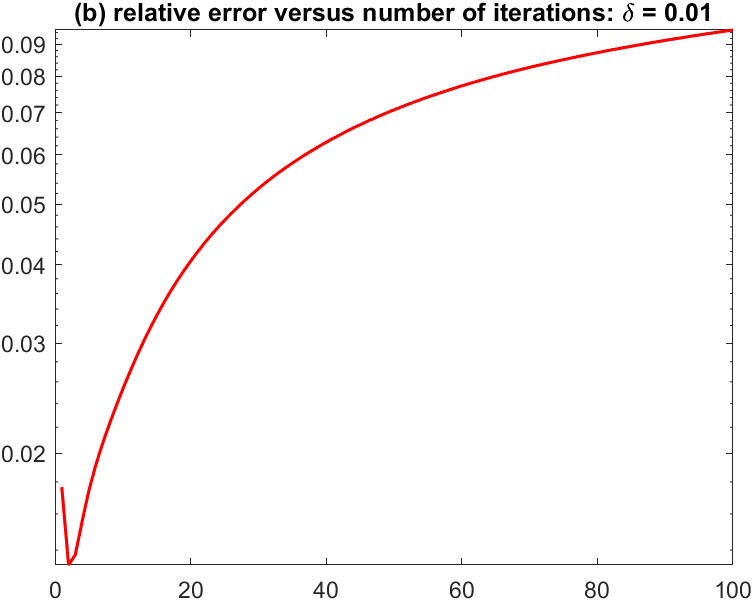}
            \includegraphics[width=0.31\textwidth]{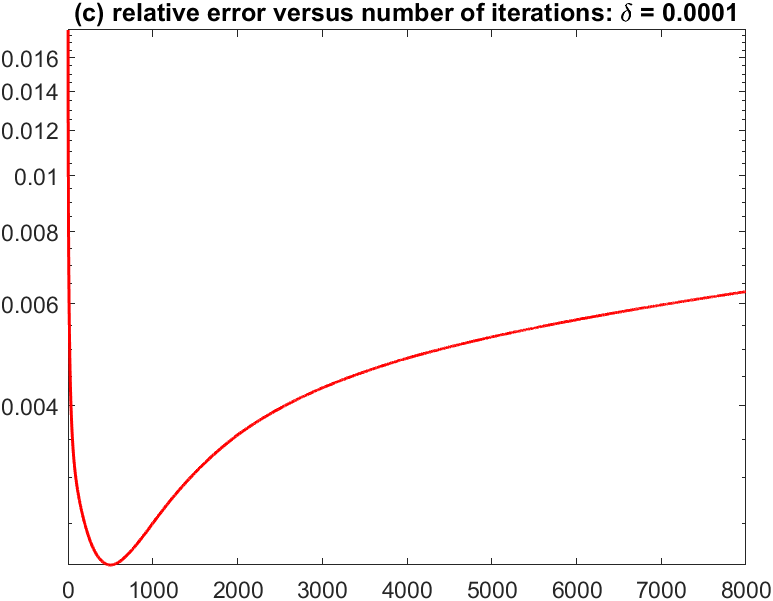}
        \caption{(a) plots the true solution $x^\dag$, (b) and (c) plot the relative errors 
        versus the number of iterations for the method (\ref{PADMM2.6}) using noisy data 
        with noise level $\d = 10^{-2}$ and 
        $10^{-4}$ respectively}\label{fig1}
\end{figure}

In order to investigate the convergence rate of the method, we generate our data as 
follows. First take $\omega^\dag \in L^2[0,1]$, set $x^\dag := \max\{A^* \omega^\dag, 0\}$ 
and define $b := A x^\dag$. Thus $x^\dag$ is a nonnegative solution of $A x = b$ satisfying 
$x^\dag = P_\C(A^* \omega^\dag)$, i.e. the source condition (\ref{sc.2}) holds. 
We use $\omega^\dag = t^3(0.9-t)(t-0.35)$, the corresponding $x^\dag$ is plotted 
in Figure \ref{fig1} (a). We then pick a random data $\xi$ with $\|\xi\|_{L^2[0,1]} = 1$ 
and generate the noisy data $b^\d$ by $b^\d := b + \d \xi$. Clearly 
$\|b^\d - b\|_{L^2[0,1]} = \d$. 

\begin{table}[ht]
\caption{Numerical results for the method (\ref{PADMM2.6}) using noisy data with diverse 
noise levels, where $\texttt{err}_{\min}$ and $\texttt{iter}_{\min}$ denote respectively 
the the smallest relative error and the required  number of iterations. } \label{table1}
\begin {center}
\begin{tabular}{|c|c|c|c|c|c|c|}
     \hline
 $\d$    &$\texttt{err}_{\min}$& $\texttt{iter}_{\min}$ &  $\texttt{err}_{\min}/\d^{1/2}$ & $\texttt{err}_{\min}/\d^{1/4}$  \\
\hline
1e-1     & 4.9307e-2  & 1      & 0.155922 & 0.087681 \\  
1e-2     & 1.3255e-2  & 2      & 0.132553 & 0.041917 \\ 
1e-3     & 5.2985e-3  & 19     & 0.167552 & 0.029796 \\ 
1e-4     & 2.1196e-3  & 501    & 0.211957 & 0.021196 \\ 
1e-5     & 7.2638e-4  & 2512   & 0.229702 & 0.012917 \\ 
1e-6     & 2.7450e-4  & 31447  & 0.274496 & 0.008680 \\ 
1e-7     & 7.4693e-5  & 329542 & 0.236199 & 0.004200 \\ 
\hline
\end{tabular}\\[5mm]
\end{center}
\end{table}

For numerical implementation, we discretize the equation by the trapzoidal rule based on 
partitioning $[0, 1]$ into $N-1$ subintervals of equal length with $N = 600$. We then 
execute the method (\ref{PADMM2.6}) with $Q =0$, $\rho_1 = 10$, $\rho_2=1$ and the 
initial guess $x^0 = \la^0 = \nu^0 = 0$ using the noisy data $b^\d$ for several distinct 
values of $\d$. In Figure \ref{fig1} (b) and (c) we plot the relative error 
$\|x^k- x^\dag\|_{L^2}/\|x^\dag\|_{L^2}$ versus $k$, the number of iterations, 
for $\d = 10^{-2}$ and $\d = 10^{-4}$ respectively. These plots demonstrate that 
the proximal ADMM always exhibits the semi-convergence phenomenon when used to solve 
ill-posed problems, no matter how small the noise level is. Therefore, properly 
terminating the iteration is important to produce useful approximate solutions. 
This has been done in \cite{JJLW2016,JJLW2017}. 

In Table \ref{table1} we report further numerical results. For the noisy data $b^\d$ 
with each noise level $\d = 10^{-i}$, $i = 1, \cdots, 7$, we execute the method and 
determine the smallest relative error, denoted by $\texttt{err}_{\min}$, and the 
required  number of iterations, denoted by $\texttt{iter}_{\min}$. The ratios 
$\texttt{err}_{\min}/\d^{1/2}$ and $\texttt{err}_{\min}/\d^{1/4}$ are then calculated. 
Since $x^\dag$ satisfies the source condition (\ref{sc.2}), our theoretical result 
predicts the convergence rate $O(\d^{1/4})$. However, Table \ref{table1} illustrates 
that the value of $\texttt{err}_{\min}/\d^{1/2}$ does not change much while the value of 
$\texttt{err}_{\min}/\d^{1/4}$ tends to decrease to 0 as $\d \to 0$. This strongly 
suggests that the proximal ADMM admits the order optimal convergence rate $O(\d^{1/2})$ 
if the source condition (\ref{sc.2}) holds. However, how to derive this order optimal 
rate remains open. 
\end{example}





\end{document}